\documentclass[12pt,letterpaper]{amsart}%
\usepackage{amsmath}
\usepackage{txfonts}
\usepackage{graphicx}
\usepackage{amsthm}
\usepackage{amsfonts}
\usepackage{amssymb}
\usepackage[utf8]{inputenc}
\usepackage{comment}
\usepackage{xcolor}
\usepackage{lipsum}
\usepackage{ifthen}%
\usepackage{tikz}

\usepackage{newtx}
\usepackage{relsize}

\providecommand{\U}[1]{\protect\rule{.1in}{.1in}}
\numberwithin{equation}{section}
\newtheorem{theorem}{Theorem}[section]
\newtheorem{lemma}[theorem]{Lemma}
\newtheorem{thm}[theorem]{Theorem}
\newtheorem{prop}[theorem]{Proposition}
\newtheorem{cor}[theorem]{Corollary}
\newtheorem{defn}[theorem]{Definition}
\newtheorem{rem}[theorem]{Remark}
\newtheorem{que}[theorem]{Question}

\newcommand{\p}{\partial}
\newcommand{\vphi}{\varphi}
\newcommand{\om}{\omega}
\newcommand{\tri}{\triangle}
\newcommand{\eps}{\epsilon}
\newcommand{\thmref}[1]{Theorem~\ref{#1}}

\def\p{\partial}

\let\vphi=\varphi

\allowdisplaybreaks

\newcommand{\chapter}{\part}

\title[Estimates for complex singular Monge-Amp\`ere equations via integral method]{Estimates for singular complex Monge-Amp\`ere equations via integral method}

\author{Yunqing Wu}  
\address{
	The Institute of Geometry and Physics\\
	University of Science and Technology of China, Hefei, Anhui, People's Republic of China}
\email{yqwu19@ustc.edu.cn}

\author{Kai Zheng}  
\address{University of Chinese Academy of Sciences, Beijing 100190, P.R. China}
\email{KaiZheng@amss.ac.cn}

\begin{document}
	\maketitle
	
	\begin{abstract} 
	In this paper, we obtain gradient estimates and Laplacian estimates for the solution to the singular complex Monge-Amp\`ere equation by applying the integral method.
	\end{abstract}

	\tableofcontents

\section{Introduction}
Let $(X,\omega_{X})$ be a compact K\"ahler manifold. Let $\varphi$ solve the following complex Monge-Amp\`ere equation 
\begin{align}
\label{complex MA Kahler version}
(\omega_{X}+\sqrt{-1}\partial \bar{\partial} \varphi)^{n}=e^{P} \omega_{X}^{n}, \quad \sup_{X} \varphi=0,
\end{align}
where $P$ is a given function. For simplicity, we write
\begin{align*}
\omega_{\varphi}=\omega_{X}+\sqrt{-1}\partial \bar{\partial} \varphi. 
\end{align*}

In this paper, we start with the following gradient estimate, by applying the integral method and using a Sobolev inequality with respect to $\omega_{\varphi}$ (instead of the one with respect to the background metric $\omega_{X}$ in \cite{ChenHe2012}) proved in \cite{guedj2024kahler}. 

\begin{theorem}
	\label{thm gradient estimate with general F}
	We assume $e^{\frac{P}{\beta}  } \in W^{1,\beta}(\omega_{X})$ for some $\beta>n$, then there exists a positive constant $C_{G}$ depending on $n,\omega_{X},\beta,\| e^{\frac{P}{\beta}  }\|_{ W^{1,\beta} (\omega_{X})} ,\|\varphi\|_{C^{0}(X)}$ such that
	\begin{align*}
	|\nabla \varphi| \leq C_{G}. 
	\end{align*}
\end{theorem}

\bigskip

A lot of work has been done on the gradient estimate of $\varphi$. By using maximum principle, Blocki \cite{Blocki2009}, Guan \cite{Guan2010} proves that
the $L^{\infty}$ estimate of $|\nabla \varphi|$ is bounded by $C^{0}$-estimate of $\varphi$, the upper bounds for $P, e^{ \frac{P}{n} }, |\nabla e^{ \frac{P}{n} }|$ and a lower bound for the bisectional curvature of $(X,\omega_{X})$, which generalises the result in \cite{blocki2003regularity} under additional assumptions: $(X,\omega_{X})$ has nonnegative bisectional curvature. By applying the integral method and Moser's iteration, Chen-He \cite{ChenHe2012} proves that  $L^{\infty}$ estimate of $|\nabla \varphi|$ can be controlled by  $C^{0}$-estimate of $\varphi$, $\|P\|_{ W^{1,p_{0}}(\omega_{X}) }$ for some $p_{0}>2n$ and the background metric $\omega_{X}$. In \cite{ChenCheng2019}, without assuming the bound on the derivatives of $e^{P}$, Chen-Cheng bounds $|\nabla \varphi|$ by the background metric $\omega_{X}$, the absolute of $P$, and an upper bound for $\int_{0}^{1} \frac{w^{2}(r)}{r} dr$, where $w$ is the modulus of continuity of $P$. By applying Alexandroff-Bakelman-Pucci maximum principle directly, Guo-Phong-Tong \cite{GuoPhongTong2024}  shows that $|\nabla \varphi|$ is bounded by an upper bound of $P$, $\|e^{ \frac{P}{n} } \|_{L^{2n}(\omega_{X})}$ and a lower bound for the bisectional curvature of $(X,\omega_{X})$. Using Green's representation, 
	 Theorem \ref{thm gradient estimate with general F} has been proved by Guo-Phong-Sturm \cite[Corollary 3]{guo2024green} under extra assumptions on $P$ besides $e^{\frac{P}{\beta}  } \in W^{1,\beta}(\omega_{X}), \beta>n$. 
Recently, Liu \cite{liu2024complex} refines the Alexandroff-Bakelman-Pucci maximum principle used in the complex setting and deduces that $|\nabla \varphi|$ can be bounded by $\|e^{P} \|_{ L^{q}(\omega_{X}) }$ for some $q>n$, $\int_{X} |\nabla P|^{n} \big( \log (|\nabla P|+1    ) \big)^{p} e^{P} \omega_{X}^{n}$ for some $p>1$ and a lower bound for the bisectional curvature of $(X,\omega_{X})$.

\bigskip

We then apply Theorem \ref{thm gradient estimate with general F} to a more general equation. 
Let $\theta$ be  a big and semipositive closed $(1,1)$-form.
We consider a priori estimates for the complex Monge-Amp\`ere equation 
\begin{align}
\label{singular equation intro}
\theta_{\vphi}^{n} :=(\theta+\sqrt{-1}  \p\bar\p\vphi)^n
=  e^{P}  \om_X^n,\quad P:=h-f+F(x,\vphi),
\end{align}
where $h$ is a given function, $f$ is the singularities we will prescribe, and  $F(x,\vphi):=-\lambda\vphi$ where $\lambda\in \mathbb R$ is a fixed constant. 

We assume that there exists an effective divisor $E$ and a small positive constant $a_0$, such that $[\theta]-a_0[E]$ is K\"ahler, which means 
$\omega_{K}:=\theta+\sqrt{-1} \partial \bar{\partial} \phi_{E}$
is a K\"ahler metric, where $\phi_{E}=a_{0} \log |s_{E}|_{ h_{E} }^{2}$, 
$h_{E}$ is a smooth Hermitian metric on the associated line bundle of $E$ and $s_{E}$ is the holomorphic section of $E$ with $|s_{E}|_{h_{E}} \leq 1$. This assumption naturally holds by the Kodaira lemma when $X$ is projective, but are unaffirmative for general case. 

The singularity term $e^{-f}$ we consider here is of divisorial type along $D$, which is a  set consisting with finite number, disjoint, smooth components in $X$ with the complex co-dimension of $n_s$, and written as 
\begin{align*}
e^{-f}=|s|_{D}^{2 \kappa}, \quad \kappa \geq 0,
\end{align*}
where $|s|_{D}^{2}$ is a smooth function defined on $X$ vanishing on $D$ and satisfying (\ref{estimate of S eps})  and (\ref{integral condition of f}). 
Intuitively, but not strictly speaking, 
one may imagine that $|s|_{D}^{2}$ behaves locally as $\sum_{1 \leq i \leq n_{s}} |z^{i}|^{2}$ where $(z^{1},\cdots,z^{n})$ is a coordinate system.  

With the above notions, we consider the (t,$\epsilon$)-approximating equation of (\ref{singular equation intro}), where $t \in (0,1), \epsilon \in (0,\frac{1}{4})$ as follows: 
	\begin{align}
	\label{approximating equation intro}
	\omega_{\vphi_{t,\epsilon}}^{n} :=(  \omega_{t}+ \sqrt{-1} \partial \bar{\partial} \varphi_{    t,\epsilon}      )^{n}=  e^{ h-f_{t,\epsilon} -\lambda\vphi_{t,\epsilon }   } \omega_X^{n}, \quad e^{-f_{t,\eps}}=S_{\epsilon}^{ \kappa}e^{c_{t,\epsilon}}, 
	\end{align}
	where $\omega_{t}=\theta+t \omega_{X}, S_{\epsilon}:=|s|_{D}^{2}+\epsilon$ and $c_{t,\epsilon}$ is chosen for normalisation:
	\begin{align*}
	[\omega_{t}]^{n}=\int_{X}  e^{ h -f_{t,\eps}-\lambda\vphi_{  t,\epsilon}    } \omega_X^{n}=e^{  c_{t,\epsilon}  } \int_{X} S_{\epsilon}^{\kappa} e^{ h-\lambda \varphi_{t,\epsilon}  } \omega_{X}^{n} .
	\end{align*}

In this paper, we always assume the existence of smooth solutions to \eqref{approximating equation intro}. It is guaranteed, thanks to Yau \cite{yau1978ricci}, when $\lambda\leq 0$. When $\lambda>0$, the solvability of \eqref{approximating equation intro} depends on properness conditions. We further assume the uniform $C^{0}$ estimate of $\varphi_{t,\epsilon}$ holds, which means $\|\varphi_{t,\epsilon} \|_{L^{\infty}(X)}$ has an finite upper bound independent of $t$ and $\epsilon$. 

Using Tsuji's trick (\cite{tsuji1988existence}, see also \cite{eyssidieux2009singular}), we rewrite (\ref{approximating equation intro}) on $E \setminus X$ as 
\begin{align}
\label{approximating after Tsuji intro}
( \tilde\omega_{t}  +\sqrt{-1}   \partial \bar{\partial}  \tilde{ \varphi}_{t,\epsilon}            )^{n}=e^{  h-\tilde f_{t,\eps}-\lambda \varphi_{t,\epsilon}  } \tilde\omega_{t}^{n}. 
\end{align}
where 
\begin{align}\label{Notations}
\tilde{ \varphi}_{t,\epsilon}  &   :=\varphi_{t,\epsilon}  -\phi_{E}  ,  \\
\tilde\omega_{t}&:=\omega_{t}+ \sqrt{-1} \partial \bar{\partial} \phi_{E}=t \omega_X+\theta+i \partial \bar{\partial} \phi_{E}=t \omega_X+\omega_{K},\\
-\tilde f_{t,\eps}&:=- f_{t,\eps} +\log \frac{ \omega_X^{n}}{\tilde\omega_{t}^{n} } =c_{t,\epsilon}+\kappa \log S_{\epsilon}+ \log \frac{ \omega_X^{n}}{\tilde\omega_{t}^{n} }. 
\end{align}
On $X \setminus E$, we can  derive the differential inequality of the gradient or Laplacian of $ \tilde{ \varphi}_{t,\epsilon} $ by standard calculation tricks. By fine properties of the test function we choose, the differential inequality holds on $X$, and 
we can still derive the uniform (independent of $t$ and $\epsilon$) gradient and Laplacian estimates of $ \tilde{ \varphi}_{t,\epsilon} $ by integral method. 

Parallel to Theorem \ref{thm gradient estimate with general F}, the following gradient estimate holds.
\begin{thm}
	\label{thm gradient estimate singular intro}
	We assume that one of the following conditions holds:\\
	(1). $e^{ \frac{h}{b} } \in W^{1,b}(\omega_{X}) $ for some $b \in (n,2n)$, $\kappa$ satisfies
	\begin{align*}
	\kappa >\frac{n}{2} - \frac{b}{2n}   n_{s};
	\end{align*}
	(2). $e^{ \frac{h}{b} } \in W^{1,b}(\omega_{X}) $ for some $b \geq 2n$, $\kappa$ satisfies
	\begin{align*}
	\kappa >\frac{n}{2} -    n_{s}.
	\end{align*}
	Then there exists a positive constant $\mathscr{C}_{\kappa}$ depending on $n,\omega_{X},\theta, \omega_{K},a_{0}, h_{E}, s_{E},$
	$\|\varphi_{t,\epsilon}\|_{L^{\infty}(X)}, b, \|e^{ \frac{h}{b} } \|_{W^{1,b}(\omega_{X}) }, 
	\kappa, n_{s}$ and a lower bound of $\|e^{h} |s|_{D}^{2 \kappa} \|_{L^{1}(\omega_{X})}$
	such that
	\begin{align}
	\label{singular gradient estimate intro}
	|\nabla \tilde{ \varphi}_{t,\epsilon} |_{  \tilde{\omega}_{t}   }  \leq  \mathscr{C}_{\kappa} |s_{E}|_{h_{E}}^{ -a_{0}( \sigma_{E}-\lambda )  }, 
	\end{align}
	where $\sigma_{E}$ is chosen large and independent of $h, \kappa, t, \epsilon$. 
\end{thm}
Up till now, we have not figured out whether Theorem \ref{thm gradient estimate singular intro} still holds if we remove the positive lower bound assumption of $\kappa$ when $n>2n_{s}$ (see Question \ref{que no angle constraints}). However, we can give an affirmative answer to this question by replacing the $W^{1,b}$ bound of $e^{ \frac{h}{b} }$ by a uniform lower bound of $\Delta_{ \tilde{\omega}_{t}} h$.  For this case, we can estimate $\Delta_{ \tilde{\omega}_{t}} \tilde{ \varphi}_{t,\epsilon} $  first and then derive the gradient estimate of $\tilde{ \varphi}_{t,\epsilon} $.

\begin{thm}
	\label{thm estimate under laplace h bound intro}
	We assume there exists $\mathscr{L} \geq 0$ such that
	$\Delta_{ \tilde{\omega}_{t}} h \geq -\mathscr{L}$ for any $t \in (0,1)$. We assume $e^{h} \in L^{p_{0}}(\omega_{X})$ for some $p_{0}>1$. Then the following estimates hold. \\
	(1). There exists a positive constant $\mathcal{L}$ depending on $n, \omega_{X}, \mathscr{L} , \|e^{h}\|_{ L^{p_{0}}(\omega_{X})   }$ and a lower bound of $\|e^{h} |s|_{D}^{2 \kappa} \|_{L^{1}(\omega_{X})}$
	such that
	\begin{align}
	\label{Laplacian estimate Laplacian h bdd intro}
	n+ \Delta_{ \tilde{\omega}_{t}} \tilde{\varphi}_{t,\epsilon} \leq \mathcal{L} |s_{E}|_{h_{E}}^{-2a_{0}K}, 
	\end{align}
	where $K$ is chosen large and independent of $h, \kappa, t, \epsilon$. \\
	(2). There exists a positive constant $\mathscr{G}$ depending on $n,\omega_{X},\theta, \omega_{K},a_{0}, h_{E}, s_{E},$
	$\|\varphi_{t,\epsilon}\|_{L^{\infty}(X)},\mathscr{L} , \|e^{h}\|_{ L^{p_{0}}(\omega_{X})   } $ and a lower bound of $\|e^{h} |s|_{D}^{2 \kappa} \|_{L^{1}(\omega_{X})}$
	such that
	\begin{align}
		\label{Gradient estimate Laplacian h bdd intro}
	|\nabla \tilde{ \varphi}_{t,\epsilon} |_{  \tilde{\omega}_{t}   }  \leq  \mathscr{G} |s_{E}|_{h_{E}}^{ -a_{0}( \sigma_{E}-\lambda )  }, 
	\end{align}
	where $\sigma_{E}$ is chosen large and independent of $h, \kappa, t, \epsilon$. 
\end{thm}
 
 In Yau's paper \cite{yau1978ricci}, the Laplacian estimate of $\varphi$ solving (\ref{complex MA Kahler version}) is obtained by using maximum principle, and $tr_{\omega_{X}} \omega_{\varphi}$ can be controlled by the upper bounds for $-\Delta_{\omega_{X}} P$ and $P$.
 	Blocki \cite{blocki2003regularity} also uses maximum principle to show that $tr_{\omega_{X}} \omega_{\varphi}$ can be bouned by upper bounds for $-e^{\frac{P}{n-1}   }\Delta_{\omega_{X}} P$ and $e^{P}$. 
 	 For the singular case, (\ref{Laplacian estimate Laplacian h bdd intro}) is proved in \cite{eyssidieux2009singular} under a stronger assumptions: $\Delta_{\tilde{\omega}_{t}} h$ has a lower bound and $e^{h}$ has an upper bound, which are exactly the assumptions raised in  \cite{yau1978ricci}. Here we weaken the conditions of the upper bound of $e^{h}$ by just assuming $L^{p}(p>1)$-integrability of $e^{h}$.

We finally derive the Laplacian estimates of $\tilde{ \varphi}_{t,\epsilon} $ by assuming only $W^{1,b}$-regularity of $e^{h}$ following the strategy in \cite{ChenHe2012}. 

\begin{thm}
	\label{thm 2nd Laplacian estimate intro}
	(1). We assume
	$e^{ \frac{h}{b} } \in W^{1,b}(\omega_{X}) $  for some $b >2n$, and $\kappa$ satisfies 
	\begin{align*}
	\kappa>(\frac{1}{2}+\frac{n_{s}}{b}     )n -n_{s}. 
	\end{align*}
	Then there exists a positive constant $\tilde{ \mathscr{C}}_{\kappa}$ depending on $n,\omega_{X},\theta, \omega_{K},a_{0}, h_{E}, s_{E},$
	$\|\varphi_{t,\epsilon}\|_{L^{\infty}(X)}, b, \|e^{ \frac{h}{b} } \|_{W^{1,b}(\omega_{X}) }, 
	\kappa, n_{s}$ and a lower bound of $\|e^{h} |s|_{D}^{2 \kappa} \|_{L^{1}(\omega_{X})}$
	such that
	\begin{align}
	\label{2nd laplacian estimate 1 intro}
	n+\Delta_{ \tilde{\omega}_{t}} \tilde{\varphi}_{t,\epsilon}  \leq \tilde{ \mathscr{C}}_{\kappa} |s_{E}|_{h_{E}}^{-2a_{0}K}, 
	\end{align}
	where $K$ is chosen large and independent of $h, \kappa, t, \epsilon$. \\
	(2). We assume $e^{ \frac{h}{b} } \in W^{1,b}(\omega_{X}) $  for some $b >2n$, $e^{ \frac{h}{\hat{b}}  } \in W^{1,\infty}(\omega_{X}) $ for some $\hat{b}>n$, and $\kappa$ satisfies 
	\begin{align*}
	\kappa>\frac{n}{2}-n_{s}. 
	\end{align*}
	Then there exists a positive constant $\hat{ \mathscr{C}}_{\kappa}$ depending on $n,\omega_{X},\theta, \omega_{K},a_{0}, h_{E}, s_{E},$
	$\|\varphi_{t,\epsilon}\|_{L^{\infty}(X)}, b, \|e^{ \frac{h}{b} } \|_{W^{1,b}(\omega_{X}) }, \hat{b}, \| e^{ \frac{h}{\hat{b}} } \|_{ W^{1,\infty}(\omega_{X}) },
	\kappa, n_{s}$ and a lower bound of $\|e^{h} |s|_{D}^{2 \kappa} \|_{L^{1}(\omega_{X})}$
	such that 
	\begin{align}
	\label{2nd laplacian estimate 2 intro}
	n+\Delta_{ \tilde{\omega}_{t}} \tilde{\varphi}_{t,\epsilon}  \leq \hat{ \mathscr{C}}_{\kappa} |s_{E}|_{h_{E}}^{-2a_{0}K}, 
	\end{align}
	where $K$ is chosen large and independent of $h, \kappa, t, \epsilon$. 
\end{thm}

 {\bf Acknowledgements}:
This work was supported by NSFC grant No. 12171365 and 12326426.
K. Zheng would like to express his deepest gratitude to IHES and the K.C. Wong Education Foundation, also the CRM and the Simons Foundation, when parts of the work were undertaken during his stays.

\section{Gradient estimate}
In this section, we will prove  Theorem \ref{thm gradient estimate with general F} and derive a gradient estimate of $\varphi$. 

	Let $H$ be an auxiliary function determined later. We 
set
\begin{align*}
v:=|\nabla \varphi|^{2}+1, \quad u:=e^{H} v. 
\end{align*}

\begin{lemma}(cf. \cite[Lemma 6.4, Lemma 6.5]{Zheng2022})
	\label{lemma rough differential inequality}
	Set $-C_{1.1}:=\inf_{X}R_{ i \bar{i} j \bar{j} }(\omega_{X})$. Then
	\begin{align*}
	u^{-1} \Delta_{\varphi} u &\geq \Delta_{\varphi} H+2 Re [  (   H_{i} +P_{i}       )     \varphi_{\bar{i}}        ] v^{-1} \\
	& + [  (-C_{1.1}-1    )  +v^{-1}     ] tr_{ \omega_{\varphi}  } \omega_{X}+( tr_{\omega_{X} } \omega_{\varphi}-n   )v^{-1}. 
	\end{align*}
\end{lemma}

\begin{prop}
	\label{prop gradient differential inequality}
	Choose $H=-\sigma \varphi$ where $\sigma=C_{1.1}+2$. Then there exists $C_{1.2}>0$ depending on $n$ and $C_{1.1}$ such that
	\begin{align}
	\label{GG1}
		\Delta_{\varphi} u \geq -C_{1.2} u+2 Re [  P_{i} \varphi_{\bar{i}}        ]  e^{H}-n e^{H}.
	\end{align} 
\end{prop}

\begin{proof}
	Since $H=-\sigma \varphi$, it yields that
	\begin{align*}
	\Delta_{\varphi} H&=-\sigma \Delta_{\varphi} \varphi=-\sigma (n-tr_{ \omega_{\varphi}  } \omega_{X}), \\
	Re [  (   H_{i} +P_{i}       )     \varphi_{\bar{i}}        ]&=Re [  (   -\sigma \varphi_{i} +P_{i}       )     \varphi_{\bar{i}}        ]=-\sigma |\nabla \varphi|^{2}+Re [  P_{i} \varphi_{\bar{i}}        ]. 
	\end{align*}
	Inserting the above into the inequality in Lemma \ref{lemma rough differential inequality}, 
	\begin{align*}
		u^{-1} \Delta_{\varphi} u &\geq -\sigma (n-tr_{ \omega_{\varphi}  } \omega_{X})+2 (-\sigma |\nabla \varphi|^{2}+Re [  P_{i} \varphi_{\bar{i}}        ] ) v^{-1} \\
	& + [  (-C_{1,1}-1    )  +v^{-1}     ] tr_{ \omega_{\varphi}  } \omega_{X}+( tr_{\omega_{X} } \omega_{\varphi}-n   )v^{-1} \\
	& \geq -\sigma n+ (\sigma-C_{1.1}-1) tr_{ \omega_{\varphi}  } \omega_{X}-2\sigma(v-1)v^{-1}+2 Re [  P_{i} \varphi_{\bar{i}}        ]  v^{-1}-n v^{-1} \\
	& \geq  -\sigma(n+2)+2 Re [  P_{i} \varphi_{\bar{i}}        ]  v^{-1}-n v^{-1},
	\end{align*}
	where we choose $\sigma \geq C_{1.1}+1$. We conclude that
	\begin{align*}
	\Delta_{\varphi} u \geq -\sigma(n+2) u+2 Re [  P_{i} \varphi_{\bar{i}}        ]  e^{H}-n e^{H}. 
	\end{align*}
\end{proof}

The following Sobolev inequality with respect to $\omega_{\varphi}$, which follows from \cite[Theorem 2.6]{guedj2024kahler}, 
plays an essential  role in our proof of Theorem \ref{thm gradient estimate with general F}. 
\begin{lemma}
	\label{lemma Sobolev}
	We assume $e^{\frac{P}{\beta}  } \in W^{1,\beta}(\omega_{X})$ for some $\beta>n$. Then for any $\gamma  \in (1,\frac{n}{n-1})$, the following Sobolev inequality with respect to the metric $\omega_{\varphi}$ holds: there exists a positive constant $C_{sob}^{\gamma}$ depending on $n, \omega_{X}, \gamma, \| e^{\frac{P}{\beta}  }\|_{ W^{1,\beta} (\omega_{X})}$ such that for any $u \in W^{1,2}(\omega_{\varphi})$, 
	\begin{align*}
\big(	\int_{X} u^{2 \gamma} \omega_{\varphi}^{n} \big)^{ \frac{1}{\gamma} } \leq C_{sob}^{\gamma} \int_{X} \big( u^{2}+|\nabla u|_{\varphi }^{2} \big) \omega_{\varphi}^{n}. 
	\end{align*}
\end{lemma}

\begin{proof}
	WLOG, we may assume $\beta \leq 2n$, otherwise $e^{\frac{P}{\beta}  } \in C^{\alpha}(X)$ for some $\alpha \in (0,1)$, and the following argument still holds. 
	By the Sobolev embedding theorem, for any $q \in [1,\frac{2n\beta}{2n-\beta})$, $\|e^{ \frac{P}{\beta} } \|_{ L^{q}(X, \omega_{X}) }$ is bounded. Since
	\begin{align*}
	\lim_{q \rightarrow \frac{2n\beta}{2n-\beta} } \frac{q}{\beta}=\frac{2n}{2n-\beta}>2, 
	\end{align*}
	then we conclude that $\|e^{P} \|_{L^{\frac{3}{2}}(\omega_{X})  }$ is bounded. By \cite[Theorem 2.6]{guedj2024kahler}, for any $\gamma \in (1,\frac{n}{n-1})$, for all $u \in W^{1,2}(\omega_{\varphi})$, 
	\begin{align*}
	\big( \int_{X} |u-\bar{u} |^{2 \gamma} \big)^{ \frac{1}{\gamma} } \leq C(\omega_{X},n,\gamma, \| e^{\frac{P}{\beta}  }\|_{ W^{1,\beta} (\omega_{X}) } )  \int_{X} |\nabla u|_{\varphi}^{2} \omega_{\varphi}^{n},
	\end{align*}
	where $\bar{u}=\frac{1}{ Vol(\omega) } \int_{X} u \omega_{\varphi}^{n}$. Since
	\begin{align*}
	u^{2 \gamma} \leq C(\gamma) ( |u-\bar{u}|^{2 \gamma}+\bar{u}^{2\gamma}), 
	\end{align*}
	we obtain
	\begin{align*}
	\int_{M} u^{2 \gamma} \omega_{\varphi}^{n} \leq & C(\gamma) \big( \int_{X} |u-\bar{u}|^{2 \gamma} \omega_{\varphi}^{n}+ \int_{X} \bar{u}^{2\gamma} \omega_{\varphi}^{n} \big)\\
	& \leq C(\omega_{X},n,\gamma, \| e^{\frac{P}{\beta}  }\|_{ W^{1,\beta} (\omega_{X}) } ) \big( \big( \int_{X} |\nabla u|_{\varphi}^{2} \omega_{\varphi}^{n} \big)^{\gamma}+\big( \int_{X} u^{2} \omega_{\varphi}^{n} \big)^{\gamma} \big), 
	\end{align*}
	which yields that
		\begin{align*}
	\big(	\int_{X} u^{2 \gamma} \omega_{\varphi}^{n} \big)^{ \frac{1}{\gamma} } \leq C_{sob}^{\gamma} \int_{M} \big( u^{2}+|\nabla u|_{\varphi }^{2} \big) \omega_{\varphi}^{n}. 
	\end{align*}
\end{proof}

Now we give the proof of Theorem \ref{thm gradient estimate with general F}. 

\begin{proof}[Proof of \thmref{thm gradient estimate with general F}]
 Multiplying both sides of  (\ref{GG1}) with $-u^{2p}$ and integrating by parts, we obtain the following integral inequality: for any $p \geq \frac{1}{2}$, 
	\begin{align}
	&\frac{2p}{(p+\frac{1}{2})^{2}   } \int_{X} |\nabla(u^{p+\frac{1}{2} })|_{\varphi}^{2} \omega_{\varphi}^{n}=-\int_{X} \Delta_{\varphi} u \cdot u^{2p} \omega_{\varphi}^{n} \nonumber \\
	\leq & \int_{X} \big( C_{1.2} u+n e^{H}-2  Re [  P_{i} \varphi_{\bar{i}}        ]  e^{H} \big) u^{2p} \omega_{\varphi}^{n} \nonumber \\
	\leq & C_{1.3} \int_{X} \big( u+1\big)u^{2p}  \omega_{\varphi}^{n} +C_{1.3} J \label{GG2}
	\end{align}
	where $C_{1.3}>0$ depends on $C_{1.2}$ and $\|\varphi\|_{C^{0}(X)}$, and
	\begin{align*}
	J=-\int_{X} Re [  P_{i} \varphi_{\bar{i}}        ]  e^{H} u^{2p} \omega_{\varphi}^{n} . 
	\end{align*}
	
	\textbf{Step 1: Dealing with the term $J$}
	 
	 By the definition of $u$ and $v$, 
	\begin{align}
	|J|& \leq \int_{X} |\nabla P|\cdot (|\nabla \varphi| e^{\frac{H}{2}}) \cdot  e^{\frac{H}{2}} u^{2p} \omega_{\varphi}^{n} \nonumber \\
	& \leq C( \|\varphi\|_{C^{0}(X) }) \int_{X} |\nabla P|\cdot v^{\frac{1}{2}} e^{\frac{H}{2}} u^{2p} \omega_{\varphi}^{n} \nonumber\\
	&=C( \|\varphi\|_{C^{0}(X) }) \int_{X} |\nabla P| u^{2p+\frac{1}{2}} \omega_{\varphi}^{n}. \label{GG3}
	\end{align}
	Choose any $\alpha \in ( \frac{\beta}{\beta-1}, \frac{n}{n-1} )$. By the H\"older inequality, 
	\begin{align}
	\int_{X} |\nabla P| u^{2p+\frac{1}{2}} \omega_{\varphi}^{n} \leq \big( \int_{X} |\nabla P|^{\lambda }\omega_{\varphi}^{n} \big)^{\frac{1}{\lambda}} \big( \int_{X} u^{(2p+1) \alpha} \omega_{\varphi}^{n} \big)^{ \frac{2p+\frac{1}{2}}{(2p+1)\alpha}   }, \label{GG4}
	\end{align}
	where 
	\begin{align*}
	\lambda=\frac{ (2p+1) \alpha}{ (2p+1) \alpha-(2p+\frac{1}{2})     }<\frac{\alpha}{\alpha-1}<\beta, \quad \forall p \geq \frac{1}{2}. 
	\end{align*}
	By the H\"older inequality, 
	\begin{align}
 \int_{X} |\nabla P|^{\lambda }\omega_{\varphi}^{n} & =\int_{X} |\nabla P|^{\lambda} e^{P} \omega_{X}^{n} \nonumber\\
 &=\int_{X} |\nabla P|^{\lambda} e^{\frac{\lambda P}{\beta}} e^{(1-\frac{\lambda }{\beta} )P }  \omega_{X}^{n} \nonumber \\
 & \leq \big( \int_{X} |\nabla P|^{\beta} e^{P} \omega_{X}^{n} \big)^{ \frac{\lambda}{\beta}  } \big( \int_{X} e^{P} \omega_{X}^{n} \big)^{ 1-\frac{\lambda}{\beta}   } \nonumber \\
 &  \leq C(\omega_{X} ) \big( \int_{X} \beta^{\beta}|\nabla e^{\frac{P}{\beta}}|^{\beta} \omega_{X}^{n} \big)^{ \frac{\lambda}{\beta}  } \nonumber \\
 & \leq C(\omega_{X} ) \beta^{\lambda} (\| e^{\frac{P}{\beta}}\|_{ W^{1,\beta} (\omega_{X})})^{\lambda}. \label{GG5}
 	\end{align}
 	Combining (\ref{GG3}), (\ref{GG4})  and (\ref{GG5}), we obtain
 	\begin{align}
 	|J| \leq C(\|\varphi\|_{C^{0}(X) }, \beta, \| e^{\frac{P}{\beta}}\|_{ W^{1,\beta} (\omega_{X})}) \big( \int_{X} u^{(2p+1) \alpha} \omega_{\varphi}^{n} \big)^{ \frac{2p+\frac{1}{2}}{(2p+1)\alpha}   }. \label{GG6}
 	\end{align}
 	
 	\textbf{Step 2: Moser iteration}
 	
 	By (\ref{GG2}) and (\ref{GG6}), there exists $C_{1.4}>0$ depending on $C_{1.3}, \|\varphi\|_{C^{0}(X) },\beta$ and $\| e^{\frac{P}{\beta}}\|_{ W^{1,\beta} (\omega_{X})}$ such that
 		\begin{align}
 		\frac{2p}{(p+\frac{1}{2})^{2}   } \int_{X} |\nabla(u^{p+\frac{1}{2} })|_{\varphi}^{2} \omega_{\varphi}^{n} \leq C_{1.4} \int_{X} ( u^{2p+1}+u^{2p} ) \omega_{\varphi}^{n}+C_{1.4} \big( \int_{X} u^{(2p+1) \alpha} \omega_{\varphi}^{n} \big)^{ \frac{2p+\frac{1}{2}}{(2p+1)\alpha}   }. \label{GG7}
 		\end{align}
 	We choose $\gamma \in (\alpha,\frac{n}{n-1})$, then by Lemma \ref{lemma Sobolev} and (\ref{GG7}),
 	\begin{align}
 	\big( \int_{X} u^{ (2p+1)\gamma } \omega_{\varphi}^{n} \big)^{\frac{1}{\gamma} } &\leq C_{sob}^{\gamma} \int_{M} \big( u^{2p+1} +|\nabla ( u^{p+\frac{1}{2}} )|_{\varphi}^{2} \big) \omega_{\varphi}^{n} \nonumber \\
 	& \leq C_{1.5}p \int_{X} ( u^{2p+1}+u^{2p} ) \omega_{\varphi}^{n}+C_{1.5}p \big( \int_{X} u^{(2p+1) \alpha} \omega_{\varphi}^{n} \big)^{ \frac{2p+\frac{1}{2}}{(2p+1)\alpha}   }, \label{GG8}
 	\end{align}
 	where $C_{1.5}=2(C_{sob}^{\gamma}+1)C_{1.4}$. 
 	By the H\"older inequality and Young's inequality, 
 	\begin{align*}
 	\int_{X} ( u^{2p+1}+u^{2p} ) \omega_{\varphi}^{n}& \leq C(\omega_{X}) \big( \big(\int_{X} u^{(2p+1)\alpha} \omega_{\varphi}^{n} \big)^{ \frac{1}{\alpha} }+\int_{X} u^{(2p+1)\alpha} \omega_{\varphi}^{n}  \big)^{ \frac{2p}{ (2p+1) \alpha}  } \big) \\
 	& \leq C(\omega_{X}) \big( 2 \big(\int_{X} u^{(2p+1)\alpha} \omega_{\varphi}^{n} \big)^{ \frac{1}{\alpha} }+1\big), \\
 	\big( \int_{X} u^{(2p+1) \alpha} \omega_{\varphi}^{n} \big)^{ \frac{2p+\frac{1}{2}}{(2p+1)\alpha}   }& 
 	\leq \big(\int_{X} u^{(2p+1)\alpha} \omega_{\varphi}^{n} \big)^{ \frac{1}{\alpha} }+1.
 	\end{align*}
 	Inserting back to (\ref{GG8}), we obtain
 	\begin{align}
 	\big( \int_{X} u^{ (2p+1)\gamma } \omega_{\varphi}^{n} \big)^{\frac{1}{\gamma} } \leq (C_{1.6}p) \big( \big(\int_{X} u^{(2p+1)\alpha} \omega_{\varphi}^{n} \big)^{ \frac{1}{\alpha} }+1 \big),  \label{GG9}
 	\end{align}
 	where $C_{1.6}$ depends on $C_{1.5}$ and $\omega_{X}$. 
 	
 	We set $w=u^{p+\frac{1}{2}}$, then
 	by (\ref{GG9}), the interpolation inequality and Young's inequality, for any $\epsilon>0$,
 	\begin{align*}
 	\|w\|_{L^{2 \gamma}(\omega_{\varphi})} &\leq (C_{1.6}p)^{\frac{1}{2}} \big( \|w\|_{L^{2 \alpha}(\omega_{\varphi})}+1 \big) \\
 	& \leq \big(C_{1.6}p)^{\frac{1}{2}} \big( \epsilon \|w\|_{L^{2 \gamma}(\omega_{\varphi})}+\epsilon^{-\frac{1-\frac{1}{2\alpha}}{ \frac{1}{2\alpha}-\frac{1}{2\gamma}   } } \|w\|_{L^{1}(\omega_{\varphi})}
 	+1 \big). 
 	\end{align*}
 	By choosing $\epsilon$ such that
 	\begin{align*}
 	(C_{1.6}p)^{\frac{1}{2}} \epsilon =\frac{1}{4}, 
 	\end{align*}
 	we conclude that
 	\begin{align*}
 	\|w\|_{L^{2 \gamma}(\omega_{\varphi})} \leq 2(C_{1.6}p)^{\frac{1}{2}}  \big( p^{\Lambda} \|w\|_{L^{1}(\omega_{\varphi})}
 	+1 \big), 
 	\end{align*}
 	where $\Lambda$ is some positive constant depending only on $\alpha$ and $\gamma$. Then we obtain
 	\begin{align}
 	\|u\|_{ L^{(2p+1)\gamma}   ( \omega_{\varphi})}&=( \|w\|_{L^{2 \gamma}(\omega_{\varphi})} )^{ \frac{1}{p+\frac{1}{2}} } \nonumber \\
 	& \leq 2(C_{1.6}p)^{\frac{1}{2p+1}} \big( p^{ \frac{\Lambda}{p+\frac{1}{2}}  } (\|w\|_{L^{1}(\omega_{\varphi})})^{ \frac{1}{p+\frac{1}{2}}    }+1 \big) \nonumber\\
 	& \leq C(\Lambda)C_{1.6} \big( \|u \|_{ L^{ p+\frac{1}{2}  }(\omega_{\varphi})    } +1 \big). \label{GG10}
 	\end{align}
 		By normalizing, we may assume $Vol(\omega_{\varphi})=Vol(\omega_{X})=1$, then $ \|u \|_{ L^{p}( \omega_{\varphi}^{n}  ) }	$ is nondecreasing with repect to $p$. 
 	If for any $p \geq \frac{1}{2}$, $\|u \|_{L^{(2p+1) \alpha}(\omega_{\varphi}^{n})  } \leq 1$, then the assertion holds immediately. Otherwise, we set 
 	\begin{align*}
 	p_{0}=\inf \{ p \geq \frac{1}{2}: \|u \|_{L^{(2p+1) \alpha}( \omega_{\varphi}^{n})  } > 1   \}. 
 	\end{align*}
 	Then for any $p <p_{0}$, we have  $\|u \|_{L^{(2p+1) \alpha}( \omega_{\varphi}^{n})  } \leq  1$. 		
 	By (\ref{GG9}), for any $p \geq p_{0}$, 
 	\begin{align*}
 	\big( \int_{X} u^{ (2p+1)\gamma } \omega_{\varphi}^{n} \big)^{\frac{1}{\gamma} } \leq (2C_{1.6}p) \big(\int_{X} u^{(2p+1)\alpha} \omega_{\varphi}^{n} \big)^{ \frac{1}{\alpha} }, 
 	\end{align*}
 	which allows us to 
 	apply the Moser iteration (the start of the iteration is by taking $p=p_{0}$). 
 	To obtain the boundness of $\|u\|_{L^{\infty}(X)}$, it suffices to bound $\|u \|_{L^{(2p_{0}+1) \alpha}( \omega_{\varphi}^{n})  } $, which follows from  (\ref{GG10}): 
 	\begin{align*}
 	\|u\|_{ L^{ (2p+1) \alpha }( \omega_{\varphi}^{n})   } \leq C(\Lambda )C_{1.6}\big( \|u\|_{ L^{  (p+\frac{1}{2})\frac{\alpha}{\gamma}  }( \omega_{\varphi}^{n})   } +1 \big) \leq 2C(\Lambda )C_{1.6}. 
 	\end{align*}
 	We complete the proof. 
\end{proof}

\section{Singular equations: Preliminary}
From  now on, we consider the following singular equations. 
Let $(X,\omega_{X})$ be a compact K\"ahler manifold and $\theta$ is a big and semipositive closed $(1,1)$-form.
In this article, we derive a priori estimates for the complex Monge-Amp\`ere equation 
\begin{align}
\label{SCMA}
\theta_{\vphi}^{n} :=(\theta+\sqrt{-1}  \p\bar\p\vphi)^n
=  e^{P}  \om_X^n,\quad P:=h-f+F(x,\vphi),
\end{align}
where $h$ is a given function, $f$ is the singularities we will prescribe, and we assume that $F(x,\vphi):=-\lambda\vphi,\quad \lambda\in \mathbb R$.

A necessary condition of the solvability of (\ref{SCMA}) is the volume constraint
\begin{align}\label{CMA normalization}
\int_Xe^{P}  \om_X^n=V:=\int_X\theta^n.
\end{align}




\subsection{Singularities/degeneration of the RHS $e^{-f}$}\label{Singularities}
Let $D$ be a set consisting with finite number, disjoint, smooth components in $X$ with the complex co-dimension of $n_s$. 

When $D$ is a divisor, $n_s=1$, then $e^{-f}$ is chosen as 
\begin{align*}
e^{-f}:=|s|_{D}^{2 \kappa} , \quad \kappa \geq 0,
\end{align*}
where $| \cdot |_{D}$ is a smooth Hermitian metric on the corresponding line bundle of $D$ and $s$ is a section of the line bundle. For $\epsilon \in (0,\frac{1}{4})$, we set $S_{\epsilon}:=|s|_{D}^{2}+\epsilon$, then by \cite[Lemma 2.11]{Zheng2022}, we have the following nice properties: there exists a nonnegative constant $C_{s}$ independent of $\epsilon$ such that
\begin{align}
\label{estimate of S eps}
C_{s} S_\eps^{-1}\om\geq  \sqrt{-1} \p\bar\p \log S_\eps\geq -C_{s} \om,\quad |\nabla S_\eps|^2 \leq  C_{s} S_\eps.
\end{align}

Inspiring by the above, for general $n_{s}$, we say $e^{-f}$ has singularities of divisorial type along $D$, if $e^{-f}$ is singular or degenerate near $D$ polynomially, that is 
\begin{align*}
e^{-f}:=|s|_{D}^{2 \kappa} \quad \kappa \geq 0, 
\end{align*}
where $|s|_{D}^{2}: X \rightarrow \mathbb{R}$ is a globally defined smooth function on $X$ vanishing on $D$, which satisfies (\ref{estimate of S eps}) and the following integral condition:
\begin{align}
\label{integral condition of f}
\int_{X} (|s|_{D}^{2})^{\alpha} \omega_{X}^{n}<+\infty, \quad \forall \alpha >-n_{s}; \quad 
\int_{X} (|s|_{D}^{2})^{\alpha} \omega_{X}^{n}=+\infty, \quad \forall  \alpha \leq -n_{s}.  
\end{align}

\subsection{The function $h$}
When $D$ is a divisor, $h$ satisfies the cohomology condition
\begin{align*}
Ric(\om_X)- \sqrt{-1} \p\bar\p(\log |s|_{D}^{2 \kappa}  +h)=\lambda\theta
\end{align*}
where $\theta\in C_1(X,D):=C_1(X)+\kappa C_1(L_D)$,
the singular Monge-Amp\`ere equation \eqref{SCMA} we consider becomes the following equation
\begin{align}
\label{singular MA}
\theta_{\varphi}^{n}:=
(  \theta+ \sqrt{-1} \partial \bar{\partial} \varphi  )^{n}= |s|_{D}^{2 \kappa} e^{ h  -\lambda \vphi} \omega_X^{n} .
\end{align}

Actually, this equation gives $Ric(\theta_\vphi)=\lambda\theta_\vphi$, which are called singular K\"ahler-Einstein metrics in the literature.

\subsection{Main equation and convenient notations}	
By the definition of $\phi_{E}$, on $X \setminus E$, we can rewrite $(\ref{approximating equation intro})$ as 
\begin{align*}
&( \omega_{t} + \sqrt{-1} \partial \bar{\partial} \phi_{E}+ \sqrt{-1} \partial \bar{\partial} (  \varphi_{t,\epsilon}-\phi_{E}   )    )^{n}=(  \omega_{t}+ \sqrt{-1} \partial \bar{\partial} \varphi_{    t,\epsilon}      )^{n} \\
=&e^{ h-f_{t,\epsilon}-\lambda  \varphi_{t,\epsilon}   } e^{\log \omega_{X}^{n}-\log (\omega_{t} + \sqrt{-1} \partial \bar{\partial} \phi_{E})^{n}  } (\omega_{t} + \sqrt{-1} \partial \bar{\partial} \phi_{E})^{n}. 
\end{align*}
We use the notations \eqref{Notations} and set 
\begin{align*}
F_{t,\epsilon}:=-\lambda\vphi_{t,\epsilon}, \quad
 \tilde P_{t,\epsilon}&  : =h-\tilde f_{t,\eps}+F_{t,\epsilon}, 
\end{align*}
then on $X \setminus E$, \eqref{approximating equation intro} becomes
\begin{align}
\label{MA using Koriada}
	\omega_{\vphi_{t,\epsilon}}^{n}=( \tilde\omega_{t}  +\sqrt{-1}   \partial \bar{\partial}  \tilde{ \varphi}_{t,\epsilon}            )^{n}=e^{  \tilde P_{t,\epsilon}   } \tilde\omega_{t}^{n}. 
\end{align}

For simplicity, we write
\begin{align*}
\psi:=\tilde{ \varphi}_{t,\epsilon},\quad P:=\tilde P_{t,\epsilon},  \quad F:=F_{t,\epsilon} , \quad f:=\tilde f_{t,\eps}, \quad \omega:=\tilde\omega_{t}, \quad \omega_{\psi}:=\tilde\omega_{t}  +\sqrt{-1} \partial \bar{\partial}  \tilde{ \varphi}_{t,\epsilon}     . 
\end{align*}

Therefore, (\ref{MA using Koriada})  is further rewritten as the main equation we consider from now on:
\begin{align}
\label{MA using Koriada smooth}
\omega_{\psi}^n=(\omega+\sqrt{-1} \partial \bar{\partial } \psi    )^{n}=e^{P} \omega^{n},\quad P=F+h-f.
\end{align}

	We emphasis again that $h$ is a given function, $F=-\lambda(\psi+\phi_E)$ is the nonlinear term deriving from the geometric problem,  $e^{-f}$ represents the singular/degenerate term containing $$S_\eps^\kappa=(|s|_{D}^{2}+\epsilon)^\kappa, \quad \kappa \geq 0.$$ We may further assume $S_{\epsilon} \leq 1$ for any $\epsilon \in (0,\frac{1}{4})$. 

The notions $\tri, \tri_\psi$ stand for the Laplace operators with respect to the smooth metrics $\om$, $\om_\psi$, respectively.

\subsection{Useful lemmas}	

The following Sobolev inequality is a version of Lemma \ref{lemma Sobolev} under weaker assumptions, whose proof follows almost the same argument except for a slight modification. We still present its proof here for completeness. 
\begin{lemma}
	\label{lemma Sobolev inequality uniform}
	Fix $\kappa \geq 0$. 
	We assume $e^{h}  \in L^{ p_{0} }(\omega_{X}  )$ for some $p_{0}>1$.  
	Then for any $\gamma \in (1,\frac{n}{n-1})$, there exists a positive uniform constant $C_{sob}^{\gamma}$ depending on $n,\omega_{X}, p_{0}, \theta, \gamma, \| e^{h} \|_{L^{p_{0}}(\omega_{X})}$ and a lower bound of $\| e^{h} |s|_{D}^{2 \kappa} \|_{L^{1}(\omega_{X})}$
	 such that for any $u \in C^{1}(X)$,
\begin{align*}
\big( \int_{X} u^{2\gamma} \omega_{\psi}^{n} \big)^{\frac{1}{\gamma} } \leq C_{sob}^{\gamma} \int_{X}
\big( u^{2}+|\nabla u|_{\psi }^{2} \big) \omega_{\psi}^{n}.  
\end{align*}
\end{lemma}

\begin{proof}
	By the definition of $\omega_{\psi}$, 
	\begin{align*}
	[\omega_{\psi}]= [\omega_{t}]=[ \theta+t \omega_{X}  ] \leq C(\theta) [\omega_{X}]. 
	\end{align*}
	Since $\| \varphi_{t,\epsilon} \|_{L^{\infty}(X) }$ is uniformly bounded and $e^{h}   \in L^{p_{0}}( \omega_{X}^{n}  )$, then
	\begin{align*}
	\int_{X} \big( \frac{[\omega_{X} ]^{n}   }{ [ \omega_{\psi}  ]^{n}} \frac{ \omega_{\psi}^{n} }{ \omega_{X}^{n} } \big)^{p_{0}} \omega_{X}^{n}&= \int_{X} \big( \frac{[\omega_{X} ]^{n}   }{ [ \omega_{t}  ]^{n}} S_{\epsilon}^{\kappa} e^{h -\lambda \vphi_{t,\epsilon}+c_{t,\epsilon}  }
	\big)^{p_{0}} \omega_{X}^{n} \\
	&=\int_{X} \big( [\omega_{X}]^{n} \frac{ S_{\epsilon}^{\kappa}  e^{h -\lambda \vphi_{t,\epsilon} }  }{    \int_{X}  S_{\epsilon}^{\kappa}  e^{h-\lambda \vphi_{t,\epsilon} } \omega_{X}^{n}  } \big)^{p_{0}} \omega_{X}^{n}  \\
	& \leq \frac{C([\omega_{X}]^{n}, \| \varphi_{t,\epsilon} \|_{L^{\infty}(X) } )   }{  \big( \int_{X} e^{h} |s|_{D}^{2 \kappa} \omega_{X}^{n} \big)^{p_{0}}       } \int_{X} (e^{ h} )^{p_{0}} \omega_{X}^{n}.
	\end{align*}
	Applying \cite[Theorem 2.6]{guedj2024kahler}, for any $\gamma \in (1,\frac{n}{n-1})$, there exists a positive constant $\tilde{C}_{sob}^{\gamma}$ depending on $n,\omega_{X}, p_{0}, \theta, \gamma, \| e^{h} \|_{L^{p_{0}}(\omega_{X})}$ and a lower bound of $\| e^{h} |s|_{D}^{2 \kappa} \|_{L^{1}(\omega_{X})}$ such that for any $u \in C^{1}(X)$, 
\begin{align*}
\big(	\frac{1}{ V_{t} } \int_{X} |u-\bar{u} |^{2\gamma} \omega_{\psi}^{n} \big)^{\frac{1}{\gamma}} \leq \tilde{C} \frac{1}{V_{t}} \int_{X} |\nabla u|_{\psi }^{2} \omega_{\psi}^{n},
\end{align*}
where 
\begin{align*}
V_{t}=[\omega_{\psi}]^{n}=[\omega_{t}]^{n}, \quad \bar{u}=\frac{1}{ V_{t} } \int_{X} u \omega_{\psi}^{n}. 
\end{align*}
By the Minkowski inequality and the H\"older inequality, 
\begin{align*}
\int_{X} u^{2\gamma} \omega_{\psi}^{n} & \leq C(\gamma) \big( \int_{X} |u-\bar{u}|^{2\gamma} \omega_{\psi}^{n} +\int_{X} \bar{u}^{2\gamma} \omega_{\psi}^{n} \big) \\
& \leq C(\gamma,\tilde{C}_{sob}^{\gamma}) V_{t}^{ 1-\gamma  } \big( \int_{X} |\nabla u |_{\psi}^{2} \omega_{\psi}^{n} \big)^{\gamma}+C(\gamma) V_{t}^{1-\gamma}\big( \int_{X} u^{2} \omega_{\psi}^{n} \big)^{\gamma} \\
& \leq C(\gamma, \tilde{C}_{sob}^{\gamma}, \omega_{X}, \theta)  \big( \big( \int_{X} |\nabla u |_{\psi}^{2} \omega_{\psi}^{n} \big)^{\gamma}+\big( \int_{X} u^{2} \omega_{\psi}^{n} \big)^{\gamma} \big). 
\end{align*}
We complete the proof. 
\end{proof}

\begin{rem}
	\label{remark c t epsilon}
	From the proof of Lemma \ref{lemma Sobolev inequality uniform}, we know that
	\begin{align*}
	e^{ c_{t,\epsilon}   } =\frac{ [\omega_{t} ]^{n}  }{  \int_{X} S_{\epsilon}^{ \kappa} e^{  h-\lambda \varphi_{t,\epsilon}   } \omega_{X}^{n}  } \leq C( \omega_{X}, \theta, \lambda, \| \varphi_{t,\epsilon} \|_{ L^{\infty}(X) }  ) \frac{1}{ \int_{X} e^{h} |s|_{D}^{2 \kappa} \omega_{X}^{n}     } .
	\end{align*}
	Hence we make an agreement: 
	 when we say some constants or estimates depend on $	 c_{t,\epsilon}    $, it means they depend on $\omega_{X}, \theta, \lambda, \|\varphi_{t,\epsilon} \|_{ L^{\infty}(X) }$ and a lower bound of $\int_{X} e^{h} |s|_{D}^{2 \kappa} \omega_{X}^{n}$. 
\end{rem}


\section{Singular equations: Gradient estimate}
We will derive the gradient estimate of $\psi$ solving 
(\ref{MA using Koriada smooth}) in this section and the Laplacian estimate in the next section. 
For simplicity, we make an agreement on the dependence of the constants during the estimates:
if a constant or an estimate depends on at least one of $n, \omega_{X}, \theta, \omega_{K},  \lambda$, we may not write down such dependence everytime.

\subsection{Differential inequality}

Let $K$ be a positive constant determined later and $H$ be an auxiliary function on $P$ and $\psi$. We 
set
\begin{align*}
v:=|\nabla \psi|^{2}+1, \quad u:=e^{H} v. 
\end{align*}

	Let  $-C_{4.1}:=\inf_{X}R_{ i \bar{i} j \bar{j} }(\omega)$. Then by Lemma \ref{lemma rough differential inequality}, on $X \setminus E$, 
	\begin{align}
	u^{-1} \Delta_{\psi} u \geq&  \Delta_{\psi} H+2 Re [  (   H_{i} +P_{i}       )     \psi_{\bar{i}}        ] v^{-1} \nonumber \\
	& + [  (-C_{4.1}-1    )  +v^{-1}     ] tr_{ \psi  } \omega+( tr_{\omega} \omega_{\psi}-n   )v^{-1}. \label{SGE1}
	\end{align}
	
We choose
\begin{align*}
 H:=\sigma_{D} \log S_{\epsilon}+(\lambda -\sigma_{E}) \psi, 
\end{align*}
where $\sigma_{D} \in [0,1)$ and $ \sigma_{E}$ are positive constants determined later. 

By direct calculations, we estimate
 the lower bound of $ \Delta_{\psi} H$ and $2Re [  (   H_{i} +P_{i}       )     \psi_{\bar{i}}        ] v^{-1} $ in the next following two propositions. 

\begin{prop}
	\label{prop estimate of laplace H}
	On $X  \setminus E$, 
	\begin{align*}
	\Delta_{\psi} H \geq (   \sigma_{E}-\lambda+C_{4.2} \sigma_{D}   )tr_{\psi} \omega-n(\sigma_{E}-\lambda),
	\end{align*}	
	where 
	\begin{align*}
	C_{4.2}:=\inf_{ (X, \omega)} \sqrt{-1} \partial \bar{\partial } \log S_{\epsilon}. 
	\end{align*}
\end{prop}

\begin{proof}
	By the choice of $H$ and (\ref{estimate of S eps}), 
	\begin{align*}
	\Delta_{\psi} H& =		\Delta_{\psi} (\sigma_{D} \log S_{\epsilon}+(\lambda -\sigma_{E}) \psi  ) \\
	& \geq C_{4.2} \sigma_{D} tr_{\psi} \omega+(\lambda-\sigma_{E}) (n-tr_{\psi} \omega) \\
	& =(   \sigma_{E}-\lambda+C_{4.2} \sigma_{D}  )tr_{\psi} \omega-n(\sigma_{E}-\lambda).
	\end{align*}				
\end{proof}

\begin{prop}
	\label{prop estimate of HiPipsi}
		On $X  \setminus E$, 
	\begin{align*}
	2Re [  (   H_{i} +P_{i}       )     \psi_{\bar{i}}        ] v^{-1} 
	\geq  -2\sigma_{E} -A_{1} v^{-\frac{1}{2} }+2 Re[ N_{i}     \psi_{\bar{i}}     ] v^{-1},  
	\end{align*}
	where 
	\begin{align*}
	N:=(  \kappa+\sigma_{D}  )  \log S_{\epsilon}     + h  , \quad 
	A_{1}:=2 \big( |\nabla (  \log \frac{ \omega_X^{n}}{\omega^{n} }  ) |    +  |\nabla (\lambda \phi_E)|\big)   .  
	\end{align*}
\end{prop}

\begin{proof}
	By the definition of $P$, 
	\begin{align*}
	H_{i}+P_{i}& =\sigma_{D} ( \log S_{\epsilon}     )_{i} +(\lambda-\sigma_{E}) \psi_{i}+\big( h+c_{t,\epsilon}+\kappa \log S_{\epsilon}+ \log \frac{ \omega_X^{n}}{\omega^{n} }-\lambda \psi-\lambda \phi_{E} \big)_{i} \\
	& =-\sigma_{E} \psi_{i}+ (  \kappa+\sigma_{D}  ) ( \log S_{\epsilon}     )_{i} + h_{i}+\big( \log \frac{ \omega_X^{n}}{\omega^{n} }  -\lambda \phi_{E} \big)_{i},
	\end{align*}
	which yields that
	\begin{align*}
	&2Re [  (   H_{i} +P_{i}       )     \psi_{\bar{i}}        ] \\
	=& -2\sigma_{E} |\nabla \psi|^{2}+
	2 Re[ \big(   (  \kappa+\sigma_{D}  ) ( \log S_{\epsilon}     )_{i} +h_{i}+\big( \log \frac{ \omega_X^{n}}{\omega^{n} }-\lambda \phi_{E} \big)_{i}  \big)        \psi_{\bar{i}}     ] \\
	\geq & -2 \sigma_{E}|\nabla \psi|^{2}-A_{1} |\nabla \psi|
	+2 Re[ N_{i}     \psi_{\bar{i}}     ] .
	\end{align*}
	
	Therefore, using $v=|\nabla \psi |^{2}+1 \geq |\nabla \psi|^{2}$, 
	\begin{align*}
	&2Re [  (   H_{i} +P_{i}       )     \psi_{\bar{i}}        ] v^{-1} \\
	 \geq & -2\sigma_{E} -A_{1} |\nabla \psi|v^{-1}+2 Re[ N_{i}     \psi_{\bar{i}}     ] v^{-1} \\
	 \geq & -2\sigma_{E} -A_{1} v^{-\frac{1}{2} }+2 Re[ N_{i}     \psi_{\bar{i}}     ] v^{-1}. 
	\end{align*}
\end{proof}

By further estimating the constants and taking $\sigma_{E}$ large, we obtain that the following differential inequality holds on $X$. 
\begin{prop}
	\label{prop differential inequality for gradient singular}
	We take $\sigma_{E}$ such that 
	\begin{align*}
\sigma_{E}-\lambda -1, \quad 	2a_{0}(\sigma_{E}-\lambda) \geq 10,
	\end{align*}
	then $u$ and $Re[N_{i} \psi_{\bar{i}}]e^{H}$ can be extended to be $C^{3}$-smooth functions by taking zero value on $E$. Furthermore, 
	there exists a uniform positive constant $C_{d}$ depending on $C_{4.1}, C_{4.2},  \omega_{K},$ $a_{0}, \sup_{X} |  \nabla(  |s_{E}|_{h_{E}}^{2} ) | , \|\varphi_{t,\epsilon}\|_{L^{\infty}(X)}$
	such that the following differential inequality holds on $X$:
	\begin{align*}
	\Delta_{\psi} u \geq -C_{d} \big( u+1 \big)+2 Re[ N_{i}     \psi_{\bar{i}}     ] e^{H}.
	\end{align*}

\end{prop}

\begin{proof}
	
	Combining (\ref{SGE1}) with Proposition \ref{prop estimate of laplace H} and Proposition \ref{prop estimate of HiPipsi} ,we obtain that on $X \setminus E$, 
	\begin{align*}
		u^{-1} \Delta_{\psi} u \geq& \Delta_{\psi} H+2 Re [  (   H_{i} +P_{i}       )     \psi_{\bar{i}}        ] v^{-1} \nonumber  + [  (-C_{4.1}-1    )  +v^{-1}     ] tr_{ \psi  } \omega+( tr_{\omega} \omega_{\psi}-n   )v^{-1} \\
	\geq &  (   \sigma_{E}-\lambda+C_{4.2} \sigma_{D}  )tr_{\psi} \omega-n(\sigma_{E}-\lambda)-2\sigma_{E} -A_{1} v^{-\frac{1}{2} }+2 Re[ N_{i}     \psi_{\bar{i}}     ] v^{-1} \\
	&+ [  (-C_{4.1}-1    )  +v^{-1}     ] tr_{ \psi  } \omega+( tr_{\omega} \omega_{\psi}-n   )v^{-1} \\
	\geq & ( \sigma_{E}-\lambda-C_{4.1}-1 -|C_{4.2}|      )tr_{ \psi  } \omega-((n+2)\sigma_{E}-n\lambda)-A_{1} v^{-\frac{1}{2}}-nv^{-1}+2 Re[ N_{i}     \psi_{\bar{i}}     ] v^{-1} \\
	\geq & -((n+2)\sigma_{E}-n\lambda)-A_{1} v^{-\frac{1}{2}}-nv^{-1}+2 Re[ N_{i}     \psi_{\bar{i}}     ] v^{-1},
	\end{align*}
	where we choose $\sigma_{E} \geq \lambda+C_{4.1}+1+|C_{4.2}|$. 
	
	Observing that
	\begin{align}
	e^{H} &=e^{  \sigma_{D} \log S_{\epsilon}+( \lambda-\sigma_{E} ) \psi    } \nonumber \\
	&=S_{\epsilon}^{\sigma_{D}} e^{ (\lambda-\sigma_{E})(  \varphi_{t,\epsilon}-a_{0} \log |s_{E}|_{h_{E}}^{2}    )   } \nonumber \\
	& \leq S_{\epsilon}^{\sigma_{D}} e^{ (\sigma_{E}-\lambda) \|\varphi_{t,\epsilon}\|_{L^{\infty}(X)}    } |s_{E}|_{h_{E}}^{  2a_{0}(\sigma_{E} -\lambda) }  \label{SGE2},
	\end{align}
	where we use the fact that $a_{0}>0, \sigma_{E}-\lambda \geq 1$, then it yields that
	\begin{align*}
	A_{1} v^{-\frac{1}{2}} u^{ \frac{1}{2} }& = 2 \big( |\nabla (  \log \frac{ \omega_X^{n}}{\omega^{n} }  ) |    +  |\nabla (\lambda \phi_E)|\big) e^{ \frac{H}{2}  } \\
	& \leq C(  \omega_{K}, a_{0},  \sup_{X} |  \nabla(  |s_{E}|_{h_{E}}^{2} ) | , \|\varphi_{t,\epsilon}\|_{L^{\infty}(X)}  )( 1+|s_{E}|_{h_{E}}^{-2} ) |s_{E}|_{h_{E}}^{a_{0}(  \sigma_{E}-\lambda )  } \\
	& \leq  C(  \omega_{K}, a_{0},  \sup_{X} |  \nabla(  |s_{E}|_{h_{E}}^{2} ) | , \|\varphi_{t,\epsilon}\|_{L^{\infty}(X)}  ):=C_{4.3},
	\end{align*}
	where we choose $\sigma_{E}$ such that $a_{0}( \sigma_{E}-\lambda ) \geq 2$. 

	Therefore, on $X \setminus E$, 
	\begin{align*}
	\Delta_{\psi} u &\geq -((n+2)\sigma_{E}-n\lambda) u-A_{1}v^{-\frac{1}{2}} u-nv^{-1} u+2 Re[ N_{i}     \psi_{\bar{i}}     ] v^{-1}u \nonumber \\
	& \geq -((n+2)\sigma_{E}-n\lambda) u-C_{4.3} u^{\frac{1}{2}} -n e^{H}+2 Re[ N_{i}     \psi_{\bar{i}}     ] e^{H} \\
	& \geq -((n+2)\sigma_{E}-n\lambda) u-C_{4.3} (u+1) -n e^{H}+2 Re[ N_{i}     \psi_{\bar{i}}     ] e^{H}\\
	& \geq -C_{d}(u+1)+ 2 Re[ N_{i}     \psi_{\bar{i}}     ] e^{H}, 
	\end{align*}
	where $C_{d}>0$ depends on $C_{4.1},C_{4.2}$ and  $C_{4.3}$. 

 By the definition of $u$ and $\psi$, 
	\begin{align*}
	u&=S_{\epsilon}^{\sigma_{D}} e^{ (\lambda-\sigma_{E}) \varphi_{t,\epsilon} } |s_{E}|_{h_{E} }^{ 2a_{0} (\sigma_{E}-\lambda)  } ( |\nabla ( \varphi_{t,\epsilon}-a_{0} \log |s_{E}|_{h_{E}}^{2}    ) |^{2}+1     )
	\\
	&=S_{\epsilon}^{\sigma_{D}} e^{ (\lambda-\sigma_{E}) \varphi_{t,\epsilon} } |s_{E}|_{h_{E} }^{ 2a_{0} (\sigma_{E}-\lambda)  } ( |\nabla  \varphi_{t,\epsilon}-a_{0} |s_{E}|_{h_{E}}^{-2}  |\nabla( |s_{E}|_{h_{E}}^{2}  )   |^{2}+1      ), \\
	Re[N_{i} \psi_{\bar{i}}]e^{H}&=\left\langle \nabla N, \nabla \varphi_{t,\epsilon}-a_{0} |s_{E}|_{h_{E}}^{-2} \nabla( |s_{E}|_{h_{E}}^{2}  )   \right\rangle S_{\epsilon}^{\sigma_{D}}e^{  (\lambda-\sigma_{E})\varphi_{t,\epsilon}   }|s_{E}|_{h_{E}}^{2a_{0}(\sigma_{E}-\lambda)}, 
	\end{align*}
	then we can write $u$ as 
	\begin{align*}
	u&=\eta_{1} |s_{E}|_{h_{E} }^{ 2a_{0} (\sigma_{E}-\lambda)  }(  |s_{E}|_{h_{E}}^{-4} +\eta_{2}       ), \\
	Re[N_{i} \psi_{\bar{i}}]e^{H}&=(\eta_{3}+\eta_{4} |s_{E}|_{h_{E}}^{-2}) |s_{E}|_{h_{E}}^{2a_{0}(\sigma_{E}-\lambda  )  },
	\end{align*}
	where $\eta_{1}$, $\eta_{2}$, $\eta_{3}$ and $\eta_{4}$ are smooth functions defined on $X$. We choose $\sigma_{E}$ such that 
	\begin{align*}
	2a_{0}( \sigma_{E}-\lambda  ) \geq 10,
	\end{align*}
	then $u$ and $Re[N_{i} \psi_{\bar{i}}]e^{H}$ can be extended to be
	  $C^{3}$-smooth functions on $X$ by taking zero value on $E$. 
	  We still denote the functions after extension by $u$ and $Re[N_{i} \psi_{\bar{i}}]e^{H}$, respectively. 
	 Since $E$ is a closed set, 
	then for any $x \in E$, we can choose $x_{k} \in X \setminus E$ converging to $x$ and deduce that
	\begin{align*}
	\Delta_{\psi} u(x)& =\lim_{k \rightarrow +\infty} \Delta_{\psi} u(x_{k}) \\
	& \geq \lim_{k \rightarrow +\infty} \big(  -C_{d} (u(x_{k})+1)+2\big(Re[ N_{i} \psi_{\bar{i}}  ]e^{H} \big)(x_{k})    \big) \\
	&=\big(  -C_{d} (u(x)+1)+2\big(Re[ N_{i} \psi_{\bar{i}}  ]e^{H}\big)(x)    \big) . 
	\end{align*}
	Indeed, both sides in the above inequality vanish on $E$. 
	
	We complete the proof. 
\end{proof}

Applying Proposition \ref{prop differential inequality for gradient singular} and integrating by parts, 
we obtain the following integral inequality. 
\begin{prop}
	\label{prop integral inequality for gradient singular}
	For any $p \geq \frac{1}{2}$, 
	\begin{align*}
	\frac{2p}{ (p+\frac{1}{2})^{2}  } \int_{X} |\nabla ( u^{ p+\frac{1}{2} } ) |_{\psi}^{2} \omega_{\psi}^{n} 
	\leq  C_{d} \int_{X} (u^{2p+1}+u^{2p} )\omega_{\psi}^{n} +\mathcal{J}, 
	\end{align*}
	where 
	\begin{align*}
	\mathcal{J}= - 2 \int_{X}  Re[    N_i  \psi_{\bar{i}}     ]e^{H} u^{2p} \omega_{\psi}^{n}.
	\end{align*}
\end{prop}

\subsection{Estimate of $\mathcal{J}$}

Recalling that $N=(\kappa+\sigma_{D}) \log S_{\epsilon}+h$, 
by (\ref{estimate of S eps}), we obtain that
\begin{align}
|\nabla N| \leq (\kappa+\sigma_{D}) |\nabla \log S_{\epsilon}|+|\nabla h| \leq C_{s} (\kappa+\sigma_{D}) S_{\epsilon}^{-\frac{1}{2} }+|\nabla h|. \label{SGE3}
\end{align}
By the definition of $u$, (\ref{SGE2}) and (\ref{SGE3}), 
\begin{align}
|Re[    N_i  \psi_{\bar{i}}     ]e^{H} | & \leq |\nabla  N ||\nabla \psi| e^{H}  \leq |\nabla N|e^{ \frac{H}{2} } \cdot u^{\frac{1}{2} }  \nonumber \\
& \leq \big( C_{s} (\kappa+\sigma_{D}) S_{\epsilon}^{-\frac{1}{2} }+|\nabla h| \big) ( C( \|\varphi_{t,\epsilon}\|_{ L^{\infty}(X) }   ) S_{\epsilon}^{\sigma_{D} })^{\frac{1}{2} } u^{\frac{1}{2} } \nonumber \\
& \leq C_{4.4} \big( S_{\epsilon}^{ \frac{\sigma_{D}-1}{2} } + |\nabla h| \big)u^{\frac{1}{2}}, \label{SGE4}
\end{align}
where $C_{4.4}$ depends on $C_{s}, \kappa$ and $\|\varphi_{t,\epsilon}\|_{ L^{\infty}(X) }  $. 

We set
\begin{align*}
 \mathcal{J}_{1}:=\int_{X}  |\nabla h| u^{2p+\frac{1}{2}}  \omega_{\psi}^{n}, \quad \mathcal{J}_{2}:=\int_{X} S_{\epsilon}^{ \frac{\sigma_{D}-1}{2} }u^{2p+\frac{1}{2}}  \omega_{\psi}^{n}. 
\end{align*}
To obtain the estimate of $|\mathcal{J}|$, it suffices to control $\mathcal{J}_{1}$ and $\mathcal{J}_{2}$.

\subsubsection{Dealing with $\mathcal{J}_{1}$}	
\begin{prop}
	\label{prop condition on h gradient singular}
	We assume that
	$e^{ \frac{h}{b} } \in W^{1,b} (\omega_{X}) $  for some $b >n$.
	Then there exists a uniform positive $C_{h}$ depending on $ b, \| e^{\frac{h}{b} } \|_{ W^{1,b}( \omega_{X} )  }$, $\|\varphi_{t,\epsilon}\|_{L^{\infty}(X)} ,
	c_{t,\epsilon}$ such that for any $\alpha \in (\frac{b}{b-1}, \frac{n}{n-1})$, 
	\begin{align*}
	\mathcal{J}_{1} \leq C_{h} \big( \int_{X} u^{ (2p+1) \alpha   } \omega_{\psi}^{n} \big)^{\frac{2p+\frac{1}{2}}{ (2p+1) \alpha}    }.
	\end{align*}
\end{prop}

\begin{proof}
	For any $\alpha \in (\frac{b}{b-1}, \frac{n}{n-1})$, by the H\"older inequality, 
	\begin{align*}
	\mathcal{J}_{1} = \int_{X} |\nabla h| u^{ 2p+\frac{1}{2} } \omega_{\psi}^{n}  \leq \big(   \int_{X} |\nabla h|^{r} \omega_{\psi}^{n} \big)^{ \frac{1}{r} } 
	\big( \int_{X} u^{ (2p+1) \alpha  } \omega_{\psi}^{n} \big)^{ \frac{2p+\frac{1}{2}}{ (2p+1) \alpha  }   }	
	\end{align*}
	where
	\begin{align*}
r=\frac{ (2p+1) \alpha}{  (2p+1) \alpha-2p-\frac{1}{2}     } \leq \frac{\alpha}{\alpha-1} <b, \quad \forall p \geq \frac{1}{2}. 
	\end{align*}
	Recalling that
	\begin{align*}
	\omega_{\psi}^{n} = S_{\epsilon}^{ \kappa} e^{ h-\lambda\vphi_{  t,\epsilon }   +c_{t,\epsilon} } \omega_X^{n}, 
	\end{align*}
	then  
	\begin{align*}
	\int_{X} |\nabla h|^{r} \omega_{\psi}^{n} \leq C(c_{t,\epsilon},  \|\varphi_{t,\epsilon}\|_{L^{\infty}(X)}    ) \int_{X} |\nabla h|^{r} e^{h}  \omega_{X}^{n}.
	\end{align*}
	Since $r <b$, then by the H\"older inequality,
	\begin{align*}
	\int_{X} |\nabla h|^{r} e^{h} \omega_{X}^{n}& \leq \int_{X} \big(  |\nabla h|_{ \omega_{X}  } (tr_{\omega} \omega_{X})^{\frac{1}{2}}  \big)^{r} e^{h} \omega_{X}^{n}
	\\
	& \leq C(\omega_{X},\omega_{K} ) \int_{X}|\nabla h|_{ \omega_{X}  }^{ r } e^{ \frac{r}{b}h   } e^{  (1-\frac{r}{b}) h    } \omega_{X}^{n}\\
	& \leq C(\omega_{X},\omega_{K} ) \big( \int_{X} |\nabla h|_{ \omega_{X}  }^{b} e^{h} \omega_{X}^{n} \big)^{ \frac{r}{b}  } \big(\int_{X} e^{h} \omega_{X}^{n} \big)^{  1-\frac{r}{b}   } \\
	&=C(\omega_{X},\omega_{K} ) b^{r}\big( \int_{X} |\nabla e^{\frac{h}{b}}|_{\omega_{X}}^{b} \omega_{X}^{n} \big)^{ \frac{r}{b}  } \big(\int_{X} e^{h} \omega_{X}^{n} \big)^{  1-\frac{r}{b}   },
	\end{align*}
	which yields that
	\begin{align*}
	\mathcal{J}_{1} \leq C(c_{t,\epsilon},  \|\varphi_{t,\epsilon}\|_{L^{\infty}(X)}   ,b, \| e^{\frac{h}{b} } \|_{ W^{1,b}( \omega_{X} )  } )
	 \big( \int_{X} u^{ (2p+1) \alpha  } \omega_{\psi}^{n} \big)^{ \frac{2p+\frac{1}{2}}{ (2p+1) \alpha  }   }	.
	\end{align*}
\end{proof}

\begin{rem}
	\label{rem condition on h in G2 implies sobolev}
	Assume $ 	e^{ \frac{h}{b} } \in W^{1,b}(\omega_{X})$ for some $b >n$. \\
	By the Sobolev embedding theorem, 
	\begin{itemize}
		\item 	if $b \in (n,2n)$, then
		$$ e^{ \frac{h}{b} } \in L^{q}(\omega_{X}),  \quad  \forall q \in [1, \frac{2n b}{ 2n-b  }  ), $$
		which yields that $e^{h} \in L^{ q }(\omega_{X})$ for any $q \in[1,\frac{2n}{2n-b})$. 
		\item  if $b=2n$, then 
		$$ e^{ \frac{h}{b} } \in L^{q}(\omega_{X}),  \quad  \forall q \geq 1, $$
		which yields that $e^{h} \in L^{ 2 }(\omega_{X})$ by taking $q=4n$. 
		\item if $b >2n$, then there exists $\alpha\in (0,1)$ such that
		$$ e^{ \frac{h}{b} } \in C^{\alpha}(\omega_{X}), $$
		which yields that $e^{h} \in L^{\infty}(X)$. 
	\end{itemize}
	In conclusion, if $ e^{ \frac{h}{b} } \in W^{1,b}(\omega_{X})$ for some $b >n$, then there exists $p_{0}>1$ such that $e^{h} \in L^{ p_{0} }(\omega_{X}  )$. 
\end{rem}

\subsubsection{Dealing with $\mathcal{J}_{2}$}	
Recalling that $\sigma_{D} \in [0,1)$. 
\begin{prop}
	\label{prop sharp condition on kappa}
	(a).  If $ 	e^{ \frac{h}{b} } \in W^{1,b}( \omega_{X})$ for some $b \in (n,2n)$ and the pair $(\kappa, \sigma_{D})$ satisfies
	\begin{align*}
	\kappa >\frac{1-\sigma_{D}}{2} n- \frac{b}{2n}   n_{s}, 
	\end{align*}
	then there exist $\alpha \in (1,\frac{n}{n-1})$ depending on $n, \kappa, \sigma_{D}, n_{s}, b$, and $C_{A}>0$ depending on $\|\varphi_{t,\epsilon}\|_{L^{\infty}(X)}, c_{t,\epsilon}, b, \| e^{ \frac{h}{b} } \|_{ W^{1,b}(\omega_{X}) }, n, \kappa, \sigma_{D}, n_{s}$ such that
	\begin{align*}
	\mathcal{J}_{2}\leq C_{A}  \big( \int_{X} u^{(2p+1)\alpha} \omega_{\psi}^{n} \big)^{  \frac{2p+\frac{1}{2} }{ (2p+1) \alpha  }   }. 
	\end{align*}
	(b). If $ 	e^{ \frac{h}{b} } \in W^{1,b}( \omega_{X})$ for some $b \geq 2n$ and the pair $(\kappa, \sigma_{D})$ satisfies
	\begin{align*}
	\kappa>\frac{1-\sigma_{D}}{2} n-n_{s}, 
	\end{align*}
	then there exist $\alpha \in (1,\frac{n}{n-1})$ depending on $n, \kappa, \sigma_{D}, n_{s}$, and $C_{A}>0$ depending on $\|\varphi_{t,\epsilon}\|_{L^{\infty}(X)}, c_{t,\epsilon}, b, \| e^{ \frac{h}{b} } \|_{ W^{1,b}(\omega_{X}) }, n, \kappa, \sigma_{D}, n_{s}$ such that
	\begin{align*}
	\mathcal{J}_{2}\leq C_{A}  \big( \int_{X} u^{(2p+1)\alpha} \omega_{\psi}^{n} \big)^{  \frac{2p+\frac{1}{2} }{ (2p+1) \alpha  }   }. 
	\end{align*}
\end{prop}

\begin{proof}
	For some $\alpha \in (1,\frac{n}{n-1})$ determined later, by the H\"older inequality,
	\begin{align*}
	\mathcal{J}_{2}=\int_{X} S_{\epsilon}^{ \frac{\sigma_{D}-1}{2} }u^{2p+\frac{1}{2}}  \omega_{\psi}^{n} \leq \big( \int_{X} S_{\epsilon}^{ r( \frac{\sigma_{D}-1}{2}    )   }  \omega_{\psi}^{n} \big)^{\frac{1}{r}} \big( \int_{X} u^{(2p+1)\alpha} \omega_{\psi}^{n} \big)^{  \frac{2p+\frac{1}{2} }{ (2p+1) \alpha  }   },
	\end{align*}
	where
	\begin{align*}
	r=\frac{ (2p+1) \alpha}{  (2p+1) \alpha-2p-\frac{1}{2}     } \in [ \frac{ 2 \alpha }{  2 \alpha-\frac{3}{2}  }, \frac{\alpha}{\alpha-1} ), \quad \forall p \geq \frac{1}{2}. 
	\end{align*}
	Since 
	\begin{align*}
	\omega_{\psi}^{n} = S_{\epsilon}^{ \kappa} e^{ h-\lambda\vphi_{  t,\epsilon }   +c_{t,\epsilon} } \omega_X^{n} \leq C(   \|\varphi_{t,\epsilon} \|_{L^{\infty}(X)}, c_{t,\epsilon}    )S_{\epsilon}^{ \kappa} e^{h}\omega_X^{n}. 
	\end{align*}
	 it sufficies to control the following term:
	\begin{align*}
	\int_{X} S_{\epsilon}^{ r( \frac{\sigma_{D}-1}{2}    )+\kappa   } e^{h} \omega_{X}^{n}. 
	\end{align*}

	\textbf{Case a.} If $ 	e^{ \frac{h}{b} } \in W^{1,b}( \omega_{X})$ for some $b \in (n,2n)$ and the pair $(\kappa, \sigma_{D})$ satisfies
	\begin{align*}
	\kappa >\frac{1-\sigma_{D}}{2} n- \frac{b}{2n}  n_{s}, 
	\end{align*}
	then there exists small $\epsilon>0$ depending on $n, \kappa, \sigma_{D}, n_{s}$ and $b$ such that
	\begin{align*}
	\kappa+n_{s}>\frac{1-\sigma_{D}}{2} ( n+2\epsilon) +(1-\frac{b }{2n} +\epsilon)n_{s}. 
	\end{align*}
	By Sobolev embedding theorem, for any $q \in[1,\frac{2n}{2n-b})$, $e^{h} \in L^{q}(\omega_{X})$, then
	by the H\"older inequality,
	\begin{align}
	\label{SGE5}
	\int_{X} S_{\epsilon}^{ r( \frac{\sigma_{D}-1}{2}    )+\kappa   } e^{h} \omega_{X}^{n} \leq \big( \int_{X} S_{\epsilon}^{ \big( r( \frac{\sigma_{D}-1}{2}    )+\kappa \big) \frac{q}{q-1}  } \omega_{X}^{n} \big)^{1-\frac{1}{q}   } \big( \int_{X} e^{qh} \omega_{X}^{n} \big)^{ \frac{1}{q} }. 
	\end{align}
	We choose $\alpha \in ( \frac{n+\epsilon}{n+\epsilon-1}, \frac{n}{n-1}  )$ and $q \in ( \frac{2n}{2n-b+2n \epsilon}, \frac{2n}{2n-b}   )$, 
	then 
	\begin{align*} 
	r<n+\epsilon, \quad \frac{1}{q}<\frac{2n-b}{2n}+\epsilon, \quad \frac{q}{q-1}>\frac{2n}{b},
	\end{align*}
	which yields that
	\begin{align*}
&\big(	 r( \frac{\sigma_{D}-1}{2}    )+\kappa \big) \frac{q}{q-1} +n_{s}\\
=&\big(	 r( \frac{\sigma_{D}-1}{2}    )+\kappa +\frac{q-1}{q}n_{s}   \big) \frac{q}{q-1} \\
 \geq &\big(	 (n+\epsilon)( \frac{\sigma_{D}-1}{2}    )+\kappa +n_{s}-\frac{n_{s}}{q} \big) \frac{q}{q-1} \\
\geq & \big(	 (n+\epsilon)( \frac{\sigma_{D}-1}{2}    )+\frac{1-\sigma_{D}}{2} ( n+2\epsilon) +(\frac{2n-b }{2n} +\epsilon)n_{s}-\frac{n_{s}}{q}   \big) \frac{q}{q-1} \\
\geq&  \frac{n \epsilon}{b}(1-\sigma_{D})>0. 
	\end{align*}
	Therefore, 
	\begin{align*}
		\mathcal{J}_{2} \leq C_{A}  \big( \int_{X} u^{(2p+1)\alpha} \omega_{\psi}^{n} \big)^{  \frac{2p+\frac{1}{2} }{ (2p+1) \alpha  }   },
	\end{align*}
	where $C_{A}$ depends on $\|\varphi_{t,\epsilon}\|_{L^{\infty}(X)}, c_{t,\epsilon}, b, \| e^{ \frac{h}{b} } \|_{ W^{1,b}(\omega_{X}) }, n, \kappa, \sigma_{D}, n_{s}$. 
	
	\textbf{Case b.} $ 	e^{ \frac{h}{2n} } \in W^{1,2n}( \omega_{X})$  and the pair $(\kappa, \sigma_{D})$ satisfies
	\begin{align*}
	\kappa>\frac{1-\sigma_{D}}{2} n-n_{s}, 
	\end{align*}
	then there exists small $\epsilon>0$ depending on $n, \kappa, \sigma_{D}, n_{s}$  such that
	\begin{align*}
	\kappa+n_{s}>\frac{1-\sigma_{D}}{2} ( n+2\epsilon) . 
	\end{align*}
	By Sobolev embedding theorem, for any $q \geq 1$, $e^{h} \in L^{q}(\omega_{X})$. 
	We choose $\alpha \in ( \frac{n+\epsilon}{n+\epsilon-1}, \frac{n}{n-1}  )$ and $q \in ( \frac{4n_{s}}{  (1-\sigma_{D})\epsilon}, \frac{8n_{s}}{  (1-\sigma_{D})\epsilon}  )$, then (\ref{SGE5}) still holds and 
	\begin{align*}
	&\big(	 r( \frac{\sigma_{D}-1}{2}    )+\kappa \big) \frac{q}{q-1} +n_{s}\\
	\geq &\big(	 (n+\epsilon)( \frac{\sigma_{D}-1}{2}    )+\kappa +n_{s}-\frac{n_{s}}{q} \big) \frac{q}{q-1} \\
	\geq & \big(	 (n+\epsilon)( \frac{\sigma_{D}-1}{2}    )+\frac{1-\sigma_{D}}{2} ( n+2\epsilon) -\frac{n_{s}}{q}  \big) \frac{q}{q-1} \\
	\geq&  \frac{2n_{s} (1-\sigma_{D}) \epsilon }{ 8n_{s}-(1-\sigma_{D}) \epsilon   }>0. 
	\end{align*}
		Therefore, 
	\begin{align*}
	\mathcal{J}_{2} \leq C_{A}  \big( \int_{X} u^{(2p+1)\alpha} \omega_{\psi}^{n} \big)^{  \frac{2p+\frac{1}{2} }{ (2p+1) \alpha  }   },
	\end{align*}
	where $C_{A}$ depends on $\|\varphi_{t,\epsilon}\|_{L^{\infty}(X)}, c_{t,\epsilon},  \| e^{ \frac{h}{2n} } \|_{ W^{1,2n}(\omega_{X}) }, n, \kappa, \sigma_{D}, n_{s}$. 
	
	\textbf{Case c.} $ 	e^{ \frac{h}{b} } \in W^{1,b}( \omega_{X})$  for some $b>2n$ and the pair $(\kappa, \sigma_{D})$ satisfies
	\begin{align*}
	\kappa>\frac{1-\sigma_{D}}{2} n-n_{s}, 
	\end{align*}
	then there exists small $\epsilon>0$ depending on $n, \kappa, \sigma_{D}, n_{s}$  such that
	\begin{align*}
	\kappa+n_{s}>\frac{1-\sigma_{D}}{2} ( n+2\epsilon) . 
	\end{align*}
		By Sobolev embedding theorem, $e^{ \frac{h}{b} } \in C^{\alpha}(X)$ for some $\alpha \in (0,1)$, which yields that
		\begin{align*}
		 	\int_{X} S_{\epsilon}^{ r( \frac{\sigma_{D}-1}{2}    )+\kappa   } e^{h} \omega_{X}^{n} \leq C( \|e^{\frac{h}{b} }\|_{C^{0}(X)}, b ) \int_{X} S_{\epsilon}^{ r( \frac{\sigma_{D}-1}{2}    )+\kappa   } \omega_{X}^{n}. 
	\end{align*}
	By choosing $\alpha \in (\frac{n+ \epsilon}{ n+ \epsilon-1 },\frac{n}{n-1})$, 
	then $r<n+\epsilon$, which yields that
	\begin{align*}
	r ( \frac{\sigma_{D}-1}{2}    ) +\kappa +n_{s} >( n+\epsilon )  ( \frac{\sigma_{D}-1}{2}    ) +\frac{1-\sigma_{D}}{2} ( n+2\epsilon)=\frac{1-\sigma_{D}}{2} \epsilon>0. 
	\end{align*}
		Therefore, 
	\begin{align*}
	\mathcal{J}_{2} \leq C_{A}  \big( \int_{X} u^{(2p+1)\alpha} \omega_{\psi}^{n} \big)^{  \frac{2p+\frac{1}{2} }{ (2p+1) \alpha  }   },
	\end{align*}
	where $C_{A}$ depends on $\|\varphi_{t,\epsilon}\|_{L^{\infty}(X)}, c_{t,\epsilon}, b, \| e^{ \frac{h}{b} } \|_{ W^{1,b}(\omega_{X}) }, n, \kappa, \sigma_{D}, n_{s}$. 
	
	We complete the proof by combining the above three cases. 
\end{proof}

\subsection{Gradient estimate of $\psi$}

\begin{thm}
	\label{thm gradient estimate singular}
	We assume that one of the following conditions holds:\\
	(1). $e^{ \frac{h}{b} } \in W^{1,b}(\omega_{X}) $ for some $b \in (n,2n)$, the pair $(\kappa, \sigma_{D})$ satisfies
	\begin{align*}
	\kappa >\frac{1-\sigma_{D}}{2} n- \frac{b}{2n}   n_{s};
	\end{align*}
	(2). $e^{ \frac{h}{b} } \in W^{1,b}(\omega_{X}) $ for some $b \geq 2n$, the pair $(\kappa, \sigma_{D})$ satisfies
	\begin{align*}
	\kappa >\frac{1-\sigma_{D}}{2} n-    n_{s}.
	\end{align*}
	Then there exists a positive constant $\mathscr{C}_{\kappa,\sigma_{D}}$ depending on
	$n, \omega_{X}, \theta, \omega_{K}, C_{d}, C_{h}, C_{A}, $
	$\|\varphi_{t,\epsilon}\|_{L^{\infty}(X)}, \kappa, \sigma_{D}$ and a lower bound of $\|e^{h} |s|_{D}^{2 \kappa} \|_{L^{1}(\omega_{X})}$
	such that the following gradient estimate holds:
	\begin{align}
	\label{singular gradient estimate 1}
	S_{\epsilon}^{ \frac{\sigma_{D}}{2}  }  |\nabla \psi|  \leq  \mathscr{C}_{\kappa,\sigma_{D}} |s_{E}|_{h_{E}}^{ -a_{0}( \sigma_{E}-\lambda )  }, 
	\end{align}
	where $\sigma_{E}$ satisfies:
	\begin{align*}
	\sigma_{E} \geq \lambda +C_{4.1}+1+|C_{4.2}|, \quad a_{0}( \sigma_{E}-\lambda ) \geq 5. 
	\end{align*}
\end{thm}

\begin{proof}
  Since $e^{ \frac{h}{b} } \in W^{1,b}(\omega_{X}) $ for some $b>n$, then by Remark \ref{rem condition on h in G2 implies sobolev}, there exists $p_{0}>1$ such that $e^{h} \in L^{ p_{0} } (\omega_{X} )$. 
	Applying Lemma \ref{lemma Sobolev inequality uniform} and Proposition \ref{prop integral inequality for gradient singular}, for some $\gamma \in ( 1,\frac{n}{n-1} )$ determined later, 
	\begin{align*}
	\big( \int_{X} u^{(2p+1)\gamma} \omega_{\psi}^{n} \big)^{\frac{1}{\gamma} } &\leq C_{sob}^{\gamma} \int_{X}
	\big( u^{2p+1}+|\nabla u^{p+\frac{1}{2}   }|_{\psi }^{2} \big) \omega_{\psi}^{n} \\
	& \leq C_{sob}^{\gamma} \int_{X}
	 u^{2p+1}  \omega_{\psi}^{n} +C_{sob}^{\gamma} C_{d} (2p) \int_{X} ( u^{2p+1}+u^{2p} )\omega_{\psi}^{n} +C_{sob} (2p) \mathcal{J} \\
	 & \leq C( C_{sob}^{\gamma}, C_{d}  )p \big( \int_{X} (u^{2p+1}+u^{2p}) \omega_{\psi}^{n} +\mathcal{J}  \big). 
	\end{align*}
	By our assumptions, we can apply Proposition \ref{prop condition on h gradient singular} and Proposition \ref{prop sharp condition on kappa} and obtain that
	\begin{align*}
		\big( \int_{X} u^{(2p+1)\gamma} \omega_{\psi}^{n} \big)^{\frac{1}{\gamma} } & \leq C( C_{sob}^{\gamma}, C_{d}  )p \big( \int_{X} (u^{2p+1}+u^{2p}) \omega_{\psi}^{n} +(C_{h}+C_{A}) \big( \int_{X} u^{(2p+1)\alpha} \omega_{\psi}^{n} \big)^{  \frac{2p+\frac{1}{2} }{ (2p+1) \alpha  }   } \big),
	\end{align*}
	where $\alpha \in (1,\frac{n}{n-1})$ is chosen as in the proof of Proposition \ref{prop sharp condition on kappa}. 
	
	Now we choose $\gamma  \in (\alpha,\frac{n}{n-1})$ and obtain
	\begin{align*}
	\big( \int_{X} u^{(2p+1)\gamma} \omega_{\psi}^{n} \big)^{\frac{1}{\gamma} }  \leq C(C_{sob}^{\gamma}, C_{d} ,C_{h}, C_{A}) p \big(  \int_{X} (u^{2p+1}+u^{2p}) \omega_{\psi}^{n}+ \big( \int_{X} u^{(2p+1)\alpha} \omega_{\psi}^{n} \big)^{  \frac{2p+\frac{1}{2} }{ (2p+1) \alpha  }   } \big).
	\end{align*}
	
	By the H\"older inequality and Young's inequality, 
	\begin{align*}
	\int_{X} (u^{2p+1}+u^{2p}) \omega_{\psi}^{n} &\leq C(\omega_{X}, \theta) \big(  \big( \int_{X} u^{(2p+1) \alpha} \omega_{\psi}^{n} \big)^{  \frac{1}{ \alpha}  }+  \big( \int_{X} u^{(2p+1) \alpha} \omega_{\psi}^{n} \big)^{  \frac{2p}{(2p+1) \alpha}  } \big) , \\
	& \leq C(\omega_{X}, \theta) \big(  \big( \int_{X} u^{(2p+1) \alpha} \omega_{\psi}^{n} \big)^{  \frac{1}{ \alpha}  }+ 1 \big), 
	\\
	\big( \int_{X} u^{(2p+1)\alpha} \omega_{\psi}^{n} \big)^{  \frac{2p+\frac{1}{2}}{  (2p+1)\alpha} } &\leq \big( \int_{X} u^{(2p+1) \alpha} \omega_{\psi}^{n} \big)^{  \frac{1}{ \alpha} } +1.
	\end{align*}
	Therefore, we obtain the following iteration inequality:
	\begin{align*}
	\big( \int_{X} u^{(2p+1)\gamma} \omega_{\psi}^{n} \big)^{\frac{1}{\gamma} }  \leq C(C_{sob}^{\gamma}, C_{d} ,C_{h}, C_{A} ,\omega_{X},\theta) p \big(  \big( \int_{X} u^{(2p+1) \alpha} \omega_{\psi}^{n} \big)^{  \frac{1}{ \alpha}  }+ 1 \big). 
	\end{align*}
	Apply the same argument as \textbf{Step 2} in the proof of Theorem \ref{thm gradient estimate with general F} (see (\ref{GG9}) and (\ref{GG10}) especially), we conclude that
	\begin{align*}
	\|u\|_{ L^{\infty}(X) } \leq C(C_{sob}^{\gamma}, C_{d} , C_{h}, C_{A} ,\omega_{X},\theta). 
	\end{align*}
	Therefore, 
	\begin{align*}
	S_{\epsilon}^{ \frac{\sigma_{D}}{2}  }  |\nabla \psi| & \leq  e^{ \frac{ \sigma_{E}-\lambda}{2} (  \varphi_{t,\epsilon}-a_{0} \log |s_{E}|_{h_{E}}^{2} )     } \big(\|u\|_{ L^{\infty}(X) }\big)^{\frac{1}{2}} \\
	& \leq  C(C_{sob}^{\gamma}, C_{d} ,C_{h}, C_{A} ,\omega_{X},\theta,\|\varphi_{t,\epsilon}\|_{L^{\infty}(X)} ) |s_{E}|_{h_{E}}^{ -a_{0}( \sigma_{E}-\lambda )  }. 
	\end{align*}
	We complete the proof. 
\end{proof}

Theorem \ref{thm gradient estimate singular intro} follows from  Theorem \ref{thm gradient estimate singular} by taking $\sigma_{D}=0$. 

\begin{que}
	\label{que no angle constraints}
	We assume that $e^{ \frac{h}{b} } \in W^{1,b}(\omega_{X}) $ for some $b>n$ (or replacing some weaker assumptions on $h$). Does the estimate $(\ref{singular gradient estimate intro})$ holds for all $\kappa > 0$?
	\end{que}


\section{Singular equations: Laplacian estimate}
In this section, we aim to prove Theorem \ref{thm estimate under laplace h bound intro} and Theorem \ref{thm 2nd Laplacian estimate intro}. 

\subsection{Differential inequality}

We set 
\begin{align*}
w:=e^{-K \psi} tr_{\omega} \omega_{\psi},
\end{align*}
where $K>0$ is determined later.

\begin{lemma}
	\label{lemma differential inequality for laplacian estimate}
	By choosing $K \geq C_{4.1}+1+|\lambda|+\frac{5}{a_{0}}$, on $X$, $w$ can be extended to be a $C^{3}$-smooth function on $X$ by taking zero value on $E$. Furthermore, 
	there exists a positive constant $C_{5.1}$ depending on $n,\kappa, C_{s}, \sup_{X} |\Delta\big( \log \frac{\omega_{X}^{n}}{  \omega^{n}    }\big)|,\lambda, \theta, \omega_{K},$ $ \|\varphi_{t,\epsilon} \|_{L^{\infty}(X)}  $ and $\lambda$ such that
	the following differential inequality holds:
	\begin{align}
	\label{LE1}
	\Delta_{\psi} w  \geq w tr_{\psi} w+e^{-K \psi} \Delta h -C_{5.1}(w+1   ). 
	\end{align}
\end{lemma}

\begin{proof}
	By Yau's computation \cite{yau1978ricci}, on $X \setminus E$, 
	\begin{align*}
	\Delta_{\psi} \log w& =\Delta_{\psi} \log (tr_{\omega} \omega_{\psi})-K \Delta_{\psi} \psi \\
	& \geq (K-C_{4.1}) tr_{\psi} \omega+\frac{\Delta P}{tr_{\omega} \omega_{\psi}}-Kn.
	\end{align*}
	By choosing $K\geq C_{4.1}+1$, we obtain that
	\begin{align*}
	\Delta_{\psi} w & \geq w 	\Delta_{\psi} \log w \geq w tr_{\psi} w+e^{-K \psi} \Delta P-Kn w. 
	\end{align*}
	Recalling that
	\begin{align*}
	P=h+c_{t,\epsilon}+\kappa \log S_{\epsilon}+\log \frac{\omega_{X}^{n}}{   \omega^{n}   }-\lambda \varphi_{t,\epsilon}, 
	\end{align*}
	we use (\ref{estimate of S eps}) and obtain that
	\begin{align*}
	\Delta P& \geq \Delta h-n \kappa C_{s} +\Delta\big( \log \frac{\omega_{X}^{n}}{  \omega^{n}   }\big) -\lambda \Delta( \psi+ \phi_{E} ) \\
	& \geq \Delta h-\lambda \Delta \psi-C( \kappa, C_{s}, \sup_{X} |\Delta\big( \log \frac{\omega_{X}^{n}}{  \omega^{n}   }\big)|,h_{E}   ) \\
	& \geq \Delta h-\lambda tr_{\omega} \omega_{\psi}-C( \kappa, C_{s}, \sup_{X} |\Delta\big( \log \frac{\omega_{X}^{n}}{\omega^{n}    }\big)|, h_{E}   ) . 
	\end{align*}
	Observing that
	\begin{align}
	\label{LE2}
	e^{-K \psi }=e^{-K(  \varphi_{t,\epsilon}-a_{0} \log |s_{E} |_{ h_{E} }^{2}   )  } \leq e^{  K \|\varphi_{t,\epsilon} \|_{L^{\infty}(X)}  } |s_{E}|_{h_{E}}^{2a_{0}K}, 
	\end{align}
	then we further choose $K \geq  C_{4.1}+1+|\lambda|   $ and conclude that there exists a positive constant $C_{5.1}$ depending on $\kappa, C_{s}, \sup_{X} |\Delta\big( \log \frac{\omega_{X}^{n}}{ \omega^{n}  }\big)|,\lambda, h_{E},  \|\varphi_{t,\epsilon} \|_{L^{\infty}(X)}  $ and $\lambda$ such that
	\begin{align*}
	\Delta_{\psi} w   \geq w tr_{\psi} w+e^{-K \psi} \Delta h -C_{5.1}(w+1   ). 
	\end{align*}
	By the choice of $w$ and $K$, 
	\begin{align*}
	w=e^{-K \varphi_{t,\epsilon} } tr_{\omega} \omega_{\psi} |s_{E}|_{h_{E}}^{2a_{0}K}, \quad 2a_{0}K \geq 10,
	\end{align*}
	then $w$ can be extended to be a $C^{3}$-smooth function on $X$ by taking zero value on $E$. Then the differential inequality (\ref{LE1}) holds on $X$ by applying similar arguments as in the proof of Proposition \ref{prop differential inequality for gradient singular}.
\end{proof}

\subsection{1st Laplacian estimate and its application}
\begin{thm}
	\label{thm Laplacian estimate under Delta h bounded}
	We assume $\Delta h \geq -\mathscr{L}$ for some $\mathscr{L} \geq 0$  and $e^{h}  \in L^{ p_{0} }(\omega_{X}  )$ for some $p_{0}>1$. Then there exists a positive constants depending on $C_{5.1}$, $\mathscr{L}$, $\| e^{h}  \|_{ L^{p_{0}}(\omega_{X})  }$ and a lower bound of  $\| e^{h} |s|_{D}^{2\kappa} \|_{L^{1}(\omega_{X})}$ such that
	\begin{align*}
	n+\Delta \psi \leq \mathcal{L} |s_{E}|_{h_{E}}^{-2a_{0}K},
	\end{align*}
	where $K$ satisfies
	\begin{align*}
	K \geq C_{4.1}+1+|\lambda|, \quad 2a_{0}K \geq 10. 
	\end{align*}
\end{thm}

\begin{proof}
	Since $\Delta h \geq -\mathscr{L}$, then by (\ref{LE1}) and (\ref{LE2}), 
	\begin{align}
	\label{LE3}
	\Delta_{\psi} w  \geq -\mathscr{L}e^{-K \psi} -C_{5.1}(w+1   ) \geq -C_{5.2}(w+1),
	\end{align}
	where $C_{5.2}>0$ depending on $C_{5.1}, \mathscr{L}, K, \|\varphi_{t,\epsilon} \|_{L^{\infty}(X)}$. 
	
	Integrating by parts, 
 we obtain that for any $p \geq \frac{1}{2}$, 
	\begin{align*}
	\frac{2p}{ (p+\frac{1}{2})^{2}  } \int_{X} |\nabla ( w^{ p+\frac{1}{2} } ) |_{\psi}^{2} \omega_{\psi}^{n} \leq C_{5.2} \int_{X} (w^{2p+1}+w^{2p}) \omega_{\psi}^{n}. 
	\end{align*}
	
	By Lemma \ref{lemma Sobolev inequality uniform}, we choose $\gamma \in (1,\frac{n}{n-1})$ and conclude that
	\begin{align*}
	\big( \int_{X} w^{(2p+1)\gamma} \omega_{\psi}^{n} \big)^{\frac{1}{\gamma} } & \leq C_{sob}^{\gamma} \int_{X}
	\big( w^{2p+1}+|\nabla (w^{ p+\frac {1}{2} } ) |_{\psi }^{2} \big) \omega_{\psi}^{n} \leq C_{5.3} p \int_{X} (w^{2p+1}+w^{2p}) \omega_{\psi}^{n},
	\end{align*}
	where $C_{5.3}>0$ depends on $C_{5.2}$ and $C_{sob}^{\gamma} $. 
	By Young's inequality, 
	\begin{align*}
	w^{2p} \leq \frac{2p}{2p+1} w^{2p+1}+\frac{1}{2p+1} \leq w^{2p+1}+1, 
	\end{align*}
	which yields that
	\begin{align*}
		\big( \int_{X} w^{(2p+1)\gamma} \omega_{\psi}^{n} \big)^{\frac{1}{\gamma} } \leq 2C_{5.3} p \int_{X} ( w^{2p+1}+1) \omega_{\psi}^{n}. 
	\end{align*}
	Using Moser's iteration, we obtain 
	\begin{align*}
	\|w\|_{ L^{\infty}(X)  } \leq C(\gamma, C_{5.3}). 
	\end{align*}
	Therefore, there exists a positive constant $\mathcal{L}$ depending on $C_{5.3}, \mathscr{L},  \|\varphi_{t,\epsilon} \|_{L^{\infty}(X)}$ such that
	\begin{align*}
	n+\Delta \psi =tr_{\omega} \omega_{\psi} \leq e^{ K \psi } \|w\|_{L^{\infty}(X)  } \leq C(\gamma, C_{5.3}) e^{ K( \varphi_{t,\epsilon} -a_{0} \log |s_{E}|_{h_{E}}^{2}  )   } \leq \mathcal{L} |s_{E}|_{h_{E}}^{-2a_{0}K}.
	\end{align*}
\end{proof}

One application of Theorem \ref{thm Laplacian estimate under Delta h bounded} is that we can answer Question \ref{que no angle constraints} by adding assumptions: $\Delta h \geq -\mathscr{L}$ for some $\mathscr{L} \geq 0$. Before the proof, we give a brief explanation as follows. 

We still use the notations of $u, v, H$ as in Section 4. The reason why
the angle constraints on $\kappa$ is unavoidable in Theorem \ref{thm gradient estimate singular intro} is that when we deal with the term $\mathcal{J}$ in Section 4, we take derivatives of $N=(\kappa+\sigma_{D}) \log S_{\epsilon}+h$ directly which results in the occurrence of $S_{\epsilon}^{-\frac{1}{2}}$ in some new term (see $\mathcal{J}_{2}$ by taking $\sigma_{D}=0$), and brings us issues for getting the iteration inequality if we do not add angle constraints. 

However, if we assume $\Delta h$ has a  lower bound, then $\Delta \psi$ has an upper bound by Theorem \ref{thm Laplacian estimate under Delta h bounded}, which allows us to deal with the term $\mathcal{J}$ in an alternative way: instead of taking derivatives, we integrate by parts and all new terms occured can be estimated without adding angle constraints. 

\begin{theorem}
	\label{theorem gradient estimate with Delta h bounded}
		We assume that $e^{h} \in L^{p_{0}}(\omega_{X})$ for some $p_{0}>1$ and
		$\Delta h \geq -\mathscr{L}$ for some $\mathscr{L} \geq 0$. Then there exists a positive constant $\mathscr{G}$ depending on $C_{d}, \sup_{X}  |\nabla (\log \frac{\omega_{X}^{n}}{ \omega^{n}  }  )|,\sup_{X} |\nabla (|s_{E}|_{h_{E}}^{2} )| , \|\varphi_{t,\epsilon}\|_{L^{\infty}(X)}, \mathscr{L}$, $\| e^{h}  \|_{ L^{p_{0}}(\omega_{X})  }$ and a lower bound of  $\| e^{h} |s|_{D}^{2\kappa} \|_{L^{1}(\omega_{X})}$
		such that
		\begin{align*}
		|\nabla \psi| \leq \mathscr{G}  |s_{E}|_{h_{E}}^{ -a_{0}( \sigma_{E}-\lambda )  }. 
		\end{align*}
\end{theorem}

\begin{proof}
	By Proposition \ref{prop integral inequality for gradient singular}, 
	for any $p \geq \frac{1}{2}$, 
	\begin{align}
	\label{LE4}
	\frac{2p}{ (p+\frac{1}{2})^{2}  } \int_{X} |\nabla ( u^{ p+\frac{1}{2} } ) |_{\psi}^{2} \omega_{\psi}^{n} 
	\leq  C_{d} \int_{X} (u^{2p+1}+u^{2p} )\omega_{\psi}^{n} +\mathcal{J}, 
	\end{align}
where
\begin{align*}
\mathcal{J}= - 2 \int_{X}  Re[    N_i  \psi_{\bar{i}}     ]e^{H} u^{2p} \omega_{\psi}^{n}, \quad N =(  \kappa+\sigma_{D}  )  \log S_{\epsilon}     + h. 
\end{align*}
We take $\sigma_{D}=0$ here. 
Since $ Re[    N_i  \psi_{\bar{i}}     ]e^{H} u^{2p}=0$ on $E$ by extension (see Proposition \ref{prop differential inequality for gradient singular}), then
		\begin{align*}
		\mathcal{J}= - 2 \int_{X \setminus E}  Re[    N_i  \psi_{\bar{i}}     ]e^{H} u^{2p} \omega_{\psi}^{n}. 
		\end{align*}
By the definition of $P$, 
\begin{align*}
&-\int_{X \setminus E} \left\langle \nabla (N-P), \nabla \psi \right\rangle e^{H} u^{2p} \omega_{\psi}^{n} \\
 \leq &  \int_{X \setminus E} |\nabla \big( -c_{t,\epsilon}-\log \frac{\omega_{X}^{n}}{ \omega^{n}  }+\lambda \psi +\lambda a_{0} \log |s_{E}|_{h_{E}}^{2} \big) |\cdot |\nabla \psi|e^{H} u^{2p} \omega_{\psi}^{n} \\
\leq & \int_{X \setminus E}  \big( |\nabla (\log \frac{\omega_{X}^{n}}{ \omega^{n}  }  )|+|\lambda||\nabla \psi|+|\lambda |a_{0} |\nabla (|s_{E}|_{h_{E}}^{2} )| \cdot |s_{E}|_{h_{E}}^{-2}   \big)\cdot |\nabla \psi|e^{H}u^{2p} \omega_{\psi}^{n} \\
\leq &  |\lambda | \int_{X} u^{2p+1} \omega_{\psi}^{n}+C( \sup_{X}  |\nabla (\log \frac{\omega_{X}^{n}}{ \omega^{n}  }  )|,\sup_{X} |\nabla (|s_{E}|_{h_{E}}^{2} )|    )\int_{X} |s_{E}|_{h_{E}}^{-2} e^{ \frac{H}{2}  } u^{2p+\frac{1}{2}} \omega_{\psi}^{n}. 
\end{align*}
Since
\begin{align*}
e^{\frac{H}{2} }=e^{  \frac{  \lambda-\sigma_{E}}{2} (  \varphi_{t,\epsilon}-a_{0} \log |s_{E}|_{h_{E}}^{2}  )  } \leq e^{  (\sigma_{E}-\lambda) \|\varphi_{t,\epsilon} \|_{L^{\infty}(X)}       } |s_{E}|_{h_{E}}^{   a_{0}(\sigma_{E}-\lambda)  },
\end{align*}
then by  choosing  $\sigma_{E}$ such that $a_{0}(\sigma_{E}-\lambda) \geq 2$, we conclude that
\begin{align}
\label{LE5}
-\int_{X} \left\langle \nabla (N-P), \nabla \psi \right\rangle e^{H} u^{2p} \omega_{\psi}^{n} \leq C_{5.4} \int_{X} (  u^{2p+1} +u^{2p+\frac{1}{2} } ) \omega_{\psi}^{n}, 
\end{align}
where $C_{5.4}>0$ depends on $\sup_{X}  |\nabla (\log \frac{\omega_{X}^{n}}{ \omega^{n}  }  )|, \sup_{X} |\nabla (|s_{E}|_{h_{E}}^{2})|, \|\varphi_{t,\epsilon} \|_{L^{\infty}(X)}    $. 
Inserting (\ref{LE5}) into (\ref{LE4}), we obtain that
\begin{align}
\label{LE6}
\frac{2p}{ (p+\frac{1}{2})^{2}  } \int_{X} |\nabla ( u^{ p+\frac{1}{2} } ) |_{\psi}^{2} \omega_{\psi}^{n} \leq (C_{d}+C_{5.4})  \int_{X} (u^{2p+1}+u^{2p+\frac{1}{2}}+ u^{2p}) \omega_{\psi}^{n} +\tilde{\mathcal{J} } , 
\end{align}
where 
\begin{align*}
\tilde{\mathcal{J} } =-\int_{X \setminus E} \left\langle \nabla P, \nabla \psi \right\rangle e^{H} u^{2p} \omega_{\psi}^{n} . 
\end{align*}

By direct calculations, 
\begin{align*}
\tilde{\mathcal{J} } &=-\int_{X \setminus E } \left\langle \nabla P, \nabla \psi \right\rangle e^{H} u^{2p} e^{P} \omega^{n} \\
& =-\int_{X \setminus E} \left\langle  \nabla e^{P}, \nabla \psi   \right\rangle e^{(\lambda-\sigma_{E}) \psi} u^{2p}  \omega^{n} \\
&=\frac{1}{\sigma_{E}-\lambda}  \int_{X \setminus E} \left\langle  \nabla e^{P}, \nabla e^{(\lambda-\sigma_{E}) \psi}   \right\rangle  u^{2p}  \omega^{n} \\
&=\frac{1}{\sigma_{E}-\lambda} \int_{X \setminus E} \left\langle  \nabla (e^{P} u^{2p}), \nabla e^{(\lambda-\sigma_{E}) \psi}   \right\rangle  \omega^{n} -\frac{1}{\sigma_{E}-\lambda}  \int_{X \setminus E} \left\langle  \nabla  u^{2p}, \nabla e^{(\lambda-\sigma_{E}) \psi}   \right\rangle e^{P} \omega^{n} .
\end{align*}
Observing that 
\begin{align*}
e^{ (\lambda-\sigma_{E}) \psi  }=e^{  (\lambda-\sigma_{E}) \varphi_{t,\epsilon}    } |s_{E}|_{h_{E}}^{2a_{0}(\sigma_{E}-\lambda) }, 
\end{align*}
then by choosing $\sigma_{E}$ such that $2a_{0}(\sigma_{E}-\lambda) \geq 10$, $e^{ (\lambda-\sigma_{E}) \psi  }$ can be extended to be a $C^{3}$-smooth function on $X$ by taking zero value on $E$, and we still denote the function after extension by $e^{ (\lambda-\sigma_{E}) \psi  }$. Then
\begin{align*}
\tilde{\mathcal{J} } 
&=\frac{1}{\sigma_{E}-\lambda} \int_{X} \left\langle  \nabla (e^{P} u^{2p}), \nabla e^{(\lambda-\sigma_{E}) \psi}   \right\rangle  \omega^{n} -\frac{1}{\sigma_{E}-\lambda}  \int_{X \setminus E} \left\langle  \nabla  u^{2p}, \nabla e^{(\lambda-\sigma_{E}) \psi}   \right\rangle e^{P} \omega^{n} \\
&=-\frac{1}{\sigma_{E}-\lambda}\int_{X} e^{P} u^{2p} \Delta(  e^{ (\lambda-\sigma_{E}) \psi  }   ) \omega^{n} -\frac{1}{\sigma_{E}-\lambda}  \int_{X \setminus E} \left\langle  \nabla  u^{2p}, \nabla e^{(\lambda-\sigma_{E}) \psi}   \right\rangle e^{P} \omega^{n} \\
&=\int_{X \setminus E} e^{P} u^{2p} \big(  e^{ (\lambda-\sigma_{E}) \psi  } \Delta \psi+( \lambda-\sigma_{E} ) e^{ (\lambda-\sigma_{E}) \psi  } |\nabla \psi|^{2}    \big)
\omega^{n} +\int_{X \setminus E} \left\langle  \nabla  u^{2p}, \nabla \psi   \right\rangle e^{(\lambda-\sigma_{E}) \psi}  e^{P} \omega^{n} \\
& =\int_{X \setminus E} u^{2p} e^{H} \Delta \psi e^{P} \omega^{n} \quad (\text{denoted by } \mathcal{I}_{1})\\
&+\int_{X \setminus E} 2p u^{2p-1} \left\langle \nabla u, \nabla \psi  \right\rangle  e^{H} e^{P} \omega^{n} \quad (\text{denoted by } \mathcal{I}_{2})\\
& +\int_{X \setminus E} u^{2p} \left\langle \nabla H, \nabla \psi   \right\rangle  e^{H} e^{P} \omega^{n} \quad (\text{denoted by } \mathcal{I}_{3}). 
\end{align*}
We next estimate $\mathcal{I}_{1},\mathcal{I}_{2}$ and $\mathcal{I}_{3}$ step by step. 

By Theorem \ref{thm Laplacian estimate under Delta h bounded}, 
\begin{align}
\mathcal{I}_{1} & \leq  \int_{X \setminus E} u^{2p} e^{H} (\Delta \psi+n) \omega_{\psi}^{n} \nonumber \\
&\leq  \mathcal{L}  \int_{X \setminus E} u^{2p} e^{H} |s_{E}|_{h_{E}}^{-2a_{0}K}   \omega_{\psi}^{n} \nonumber\\
& \leq \mathcal{L}  e^{ (\sigma_{E}-\lambda) \|\varphi_{t,\epsilon}\|_{L^{\infty}(X)}     } \int_{X \setminus E} u^{2p} |s_{E}|_{h_{E}}^{ 2a_{0}(\sigma_{E}-\lambda-K)      }  \omega_{\psi}^{n} , \nonumber\\
& \leq C_{5.5} \int_{X} u^{2p}  \omega_{\psi}^{n} ,\label{LE7}
\end{align}
where $C_{5.5}>0$ depends on $\mathcal{L} ,  \|\varphi_{t,\epsilon}\|_{L^{\infty}(X)}   $ and $\sigma_{E}$ is chosen such that $\sigma_{E} \geq \lambda+K$ (we should note that the choice of $K$ depends only on $C_{4.1}, \lambda$  and $a_{0}$). 

By Theorem \ref{thm Laplacian estimate under Delta h bounded} again and Young's inequality, for  any $\tau>0$, 
\begin{align}
\mathcal{I}_{2}  &\leq \int_{X \setminus E} 2p u^{2p-1} |\nabla u||\nabla \psi| e^{H} \omega_{\psi}^{n} \nonumber\\
& \leq \int_{X \setminus E} 2p u^{2p-1} |\nabla u|_{\psi} (tr_{\omega} \omega_{\psi}    )^{\frac{1}{2}}   |\nabla \psi| e^{H} \omega_{\psi}^{n} \nonumber \\
& \leq \int_{X \setminus E} 2p u^{2p-1} \big(  \tau |\nabla u|_{\psi}^{2}+C(\tau) tr_{\omega} \omega_{\psi} |\nabla \psi|^{2} e^{2H} \big) \omega_{\psi}^{n} \nonumber \\
& \leq \tau \int_{X \setminus E} \frac{2p}{  (p+\frac{1}{2})^{2}  } |\nabla (u^{p+\frac{1}{2}  }) |_{\psi}^{2} \omega_{\psi}^{n}+C(\tau) \mathcal{L} p \int_{X \setminus E} u^{2p-1}  (|s_{E}|_{h_{E}}^{ -2a_{0} K } e^{H} )(|\nabla \psi|^{2} e^{H}) \omega_{\psi}^{n} \nonumber \\
& \leq \tau \int_{X} \frac{2p}{  (p+\frac{1}{2})^{2}  } |\nabla(u^{p+\frac{1}{2}  }) |_{\psi}^{2} \omega_{\psi}^{n}+C_{5.6}p \int_{X}  u^{2p } \omega_{\psi}^{n}, \label{LE8}
\end{align}
where we also use $\sigma_{E} \geq \lambda+K$ and $C_{5.6}>0$ depends on  $\tau,\mathcal{L} ,  \|\varphi_{t,\epsilon}\|_{L^{\infty}(X)}   $. 

It is obvious that
\begin{align}
\mathcal{I}_{3} \leq \int_{X \setminus E} u^{2p} |\nabla H||\nabla \psi| e^{H} \omega_{\psi}^{n} 
 \leq (\sigma_{E}-\lambda) \int_{X \setminus E} u^{2p} |\nabla \psi|^{2} e^{H} \omega_{\psi}^{n} \leq C_{5.7} \int_{X} u^{2p+1} \omega_{\psi}^{n} . \label{LE9}
\end{align}

Inserting (\ref{LE7}), (\ref{LE8})  and (\ref{LE9}) back into (\ref{LE6}) and taking $\tau=\frac{1}{4}$, we obtain 
Therefore, we obtain
\begin{align}
\frac{2p}{ (p+\frac{1}{2})^{2}  } \int_{X} |\nabla ( u^{ p+\frac{1}{2} } ) |_{\psi}^{2} \omega_{\psi}^{n} \leq C_{5.8} p  \int_{X} (u^{2p+1}+ 1) \omega_{\psi}^{n} ,\label{LE10} 
\end{align}
where we use $u^{2p}+u^{2p+\frac{1}{2}} \leq u^{2p+1}+2$ by Young's inequality and $C_{5.8}>0$ depends on $C_{d}, C_{5.4},C_{5.5},C_{5.6}, C_{5.7}$. 

Applying  Lemma \ref{lemma Sobolev inequality uniform}, we obtain that 
\begin{align}
\big( \int_{X} u^{(2p+1)\gamma} \omega_{\psi}^{n} \big)^{\frac{1}{\gamma} }  \leq C_{5.9} p^{2} \int_{X} (u^{2p+1}+ 1) \omega_{\psi}^{n} ,\label{LE11}  
\end{align}
where $\gamma \in (1,\frac{n}{n-1})$ and $C_{5.9}$ depends on $C_{5.8}$ and $C_{sob}^{\gamma}$. 
By Moser's inequality, there exists a positive constant $C_{5.10}>1$ depending on $C_{5.9}$ such that
\begin{align*}
\|u\|_{L^{\infty}(X)} \leq C_{5.10}. 
\end{align*}
By the definition of $u$  and $C_{5.10}>1$, 
\begin{align*}
|\nabla \psi| \leq e^{ -\frac{H}{2} } u^{\frac{1}{2} }  \leq C_{5.10} e^{ \frac{ \sigma_{E}-\lambda}{2} \| \varphi_{t,\epsilon} \|_{ L^{\infty}(X) } }
|s_{E}|_{h_{E}}^{ -a_{0}( \sigma_{E}-\lambda )  }:=\mathscr{G}  |s_{E}|_{h_{E}}^{ -a_{0}( \sigma_{E}-\lambda )  }. 
\end{align*}
We complete the proof. 
\end{proof}

Theorem \ref{thm estimate under laplace h bound intro} follows from Theorem \ref{thm Laplacian estimate under Delta h bounded} and Theorem \ref{theorem gradient estimate with Delta h bounded} immediately. 

\subsection{2nd Laplacian estimate}
We divide the proof of Theorem \ref{thm 2nd Laplacian estimate intro} into the following several parts.

\subsubsection{Integral differential ineuqality}
	Multiplying both sides of  (\ref{LE1}) with $-w^{2p}$ and integrating by parts, we obtain the following integral inequality: for any $p \geq \frac{1}{2}$, 
\begin{align}
\frac{2p}{(p+\frac{1}{2})^{2}   } \int_{X} |\nabla(w^{p+\frac{1}{2} })|_{\psi}^{2} \omega_{\psi}^{n}
\leq   C_{5.1} \int_{X} (w^{2p+1}+w^{2p})  \omega_{\psi}^{n} + \mathcal{Q} \label{LE12}, 
\end{align}
where 
\begin{align*}
\mathcal{Q}:=-\int_{X} e^{-K \psi} w^{2p} \Delta h \omega_{\psi}^{n} . 
\end{align*}

\subsubsection{Dealing with $\mathcal{Q}$}
From now on, we assume $e^{ \frac{h}{b} } \in W^{1,b}(\omega_{X}) $  for some $b >2n$ and 	$\kappa >\frac{n}{2} -    n_{s}$. By Theorem \ref{thm gradient estimate singular intro}, the  gradient estimate (\ref{singular gradient estimate intro}) holds. 

Integrating by parts, 
\begin{align*}
\mathcal{Q}=&-\int_{X} e^{-K \psi} w^{2p} \Delta h e^{ P} \omega^{n} \\
=& \int_{X} \left\langle \nabla (  e^{-K \psi}w^{2p} e^{P}),  \nabla h \right\rangle   \omega^{n} \\
=&\int_{X} 2p w^{2p-1} e^{ -K \psi } \left\langle \nabla w,  \nabla h \right\rangle  e^{P} \omega^{n} \quad (\text{denoted by } \mathcal{Q}_{1}) \\
+& \int_{X} -K e^{-K \psi  } w^{2p}   \left\langle \nabla \psi,  \nabla h \right\rangle
e^{ P } \omega^{n} \quad (\text{denoted by } \mathcal{Q}_{2}) \\
+& \int_{X} e^{ -K \psi } w^{2p} \left\langle \nabla P,  \nabla h \right\rangle  e^{ P} \omega^{n} \quad (\text{denoted by } \mathcal{Q}_{3}).
\end{align*}

Oberving that $e^{-K \psi} \leq C( \|\varphi_{t,\epsilon} \|_{L^{\infty}(X)}  )|s_{E}|_{h_{E}}^{2a_{0}K}$, we apply
 Young's inequality and obtain that for any $\tau>0$, 
\begin{align}
|	\mathcal{Q}_{1}  | & \leq \int_{X} 2p w^{2p-1} e^{-K\psi} |   \nabla w|  \cdot  |\nabla h|  \omega_{\psi}^{n} \nonumber \\
& \leq \int_{X} 2p w^{2p-1} e^{-K \psi}|\nabla h|  \cdot  |\nabla w|_{\psi} (tr_{\omega} \omega_{\psi})^{\frac{1}{2}}  \omega_{\psi}^{n} \nonumber \\
& \leq  \int_{X} 2p w^{2p-1} \big( \tau|\nabla w|_{\psi}^{2}+C(\tau) e^{-2K \psi } tr_{\omega} \omega_{\psi} |\nabla h|^{2} \big)  \omega_{\psi}^{n} \nonumber\\
&=\tau \int_{X} \frac{2p}{ ( p+\frac{1}{2} )^{2} } |\nabla (w^{p+\frac{1}{2}} ) |_{\psi}^{2} \omega_{\psi}^{n} +
C_{5.11} p \mathcal{Q}_{4}, \label{LE13}
\end{align}
where $C_{5.11}$ depends on $\tau, \|\varphi_{t,\epsilon} \|_{L^{\infty}(X)} $ and 
\begin{align*}
\mathcal{Q}_{4} :=\int_{X} w^{2p} |\nabla h|^{2}  \omega_{\psi}^{n}. 
\end{align*}

By (\ref{singular gradient estimate intro}), 
\begin{align*}
|\nabla \psi| \leq \mathscr{C}_{\kappa} |s_{E}|_{h_{E}}^{ -a_{0} ( \sigma_{E}-\lambda  )      }, 
\end{align*}
hence it yields that
\begin{align}
|\mathcal{Q}_{2} | & \leq C(\|\varphi_{t,\epsilon} \|_{L^{\infty}(X)} )     \int_{X} |s_{E}|_{h_{E}}^{2a_{0}K} w^{2p} |\nabla \psi|  \cdot |\nabla h | \omega_{\psi}^{n} \nonumber \\
& \leq C(\|\varphi_{t,\epsilon} \|_{L^{\infty}(X)} )   \int_{X}  \mathscr{C}_{\kappa} |s_{E}|_{h_{E}}^{ 2a_{0}K-a_{0} ( \sigma_{E}-\lambda  )      }     w^{2p} |\nabla h| 
 \omega_{\psi}^{n}  \nonumber \\
 & \leq C(\|\varphi_{t,\epsilon} \|_{L^{\infty}(X)} ) \mathscr{C}_{\kappa} \int_{X} w^{2p}( |\nabla h|^{2}+1   ) \omega_{\psi}^{n}  \nonumber \\
 & \leq  C_{5.12} \int_{X} w^{2p} \omega_{\psi}^{n}+		C_{5.12} \mathcal{Q}_{4},  \label{LE14}
\end{align}	
where we choose $K$ such that $2K \geq \sigma_{E}-\lambda$ and $C_{5.12}$ depends on $\mathscr{C}_{\kappa}, \|\varphi_{t,\epsilon} \|_{L^{\infty}(X)}$. 

By (\ref{estimate of S eps}) and (\ref{singular gradient estimate intro}), 
\begin{align*}
|\nabla P|& = |\nabla \big(h+\kappa\log S_{\epsilon}+ \log \frac{ \omega_{X}^{n} }{ \omega^{n}  }   -\lambda \psi-\lambda a_{0} \log |s_{E}|_{h_{E}}^{2}     \big)| \\
& \leq |\nabla h|+\kappa (C_{s})^{ \frac{1}{2} } S_{\epsilon}^{-\frac{1}{2}}+|\nabla (\frac{ \omega_{X}^{n} }{ \omega^{n}  }) |+|\lambda| \mathscr{C}_{\kappa} |s_{E}|_{h_{E}}^{ -a_{0} ( \sigma_{E}-\lambda  )      }+|\lambda| a_{0} |s_{E}|_{h_{E}}^{-2} |\nabla( |s_{E}|_{h_{E}}^{2}     ) |\\
& \leq C_{5.13} \big( |\nabla h|+S_{\epsilon}^{ -\frac{1}{2} }+1+|s_{E}|_{h_{E}}^{-a_{0}(\sigma_{E}-\lambda)}+|s_{E}|_{h_{E}}^{-2} \big),
\end{align*}
where $C_{5.13}$ depends on $\kappa, C_{s} , \sup_{X} |\nabla (\frac{ \omega_{X}^{n} }{ \omega^{n}  }) |, \kappa$ and $\sup_{X} |\nabla( |s_{E}|_{h_{E}}^{2}     ) |$. 
Therefore, we obtain
\begin{align}
|\mathcal{Q}_{3}| & \leq  \int_{X} e^{-K \psi} w^{2p} |\nabla h| \cdot |\nabla P |  \omega_{\psi}^{n} \nonumber \\
&\leq C_{5.13} \int_{X} e^{-K \psi} w^{2p} |\nabla h|  \cdot \big( |\nabla h|+S_{\epsilon}^{ -\frac{1}{2} }+1+|s_{E}|_{h_{E}}^{-a_{0}(\sigma_{E}-\lambda)}+|s_{E}|_{h_{E}}^{-2} \big)  \omega_{\psi}^{n} \nonumber  \\
& \leq C_{5.14} \int_{X} w^{2p} \omega_{\psi}^{n}+C_{5.14} \mathcal{Q}_{4}+C_{5.14} \mathcal{Q}_{5} ,
 \label{LE15}
\end{align}	
where we choose $K$ such that $2a_{0} K \geq \max \{ a_{0}(\sigma_{E}-\lambda), 2   \}$, $C_{5.14}$ depends on $C_{5.13}, \| \varphi_{t,\epsilon} \|_{L^{\infty}(X)}$, and 
\begin{align*}
\mathcal{Q}_{5} :=\int_{X} w^{2p} |\nabla h|   S_{\epsilon}^{-\frac{1}{2}} \omega_{\psi}^{n}. 
\end{align*}
 Taking $\tau=\frac{1}{4}$ in (\ref{LE13}) and combining (\ref{LE12}), (\ref{LE13}), (\ref{LE14}) with (\ref{LE15}), we obtain
 \begin{align}
  \int_{X} \frac{2p}{ ( p+\frac{1}{2} )^{2} } |\nabla w^{p+\frac{1}{2}}  |_{\psi}^{2} \omega_{\psi}^{n} \leq C_{5.15} p \big( \int_{X} ( w^{2p+1}+w^{2p}  ) \omega_{\psi}^{n}+ \mathcal{Q}_{4} +\mathcal{Q}_{5}  \big), \label{LE16}
 \end{align}
where $C_{5.15}$ depends on $C_{5.1}, C_{5.11}, C_{5.12}, C_{5.14}$. 

\subsubsection{Estimate of $\mathcal{Q}_{4}$}

\begin{prop}
		\label{prop estimate of Q4}
	We assume that
	$e^{ \frac{h}{b} } \in W^{1,b}(\omega_{X}) $  for some $b >2n$.
	Then there exists a  positive constant $\tilde{C}_{h}$ depending on $\|\varphi_{t,\epsilon}\|_{L^{\infty}(X)} ,
	c_{t,\epsilon}, b, \| e^{\frac{h}{b} } \|_{ W^{1,b}( \omega_{X} )  }$ such that for any $\alpha \in (\frac{b}{b-2}, \frac{n}{n-1})$, 
	\begin{align*}
	\mathcal{Q}_{4} \leq \tilde{C}_{h} \big( \int_{X} w^{ (2p+1) \alpha   } \omega_{\psi}^{n} \big)^{\frac{2p}{ (2p+1) \alpha}    }.
	\end{align*}
\end{prop}

The proof is almost the same as  Proposition \ref{prop condition on h gradient singular}.  
\begin{proof}
	For any $\alpha \in (\frac{b}{b-2}, \frac{n}{n-1})$, by the H\"older inequality, 
	\begin{align*}
	\mathcal{Q}_{4} = \int_{X} |\nabla h|^{2} w^{ 2p} \omega_{\psi}^{n}  \leq \big(   \int_{X} |\nabla h|^{2 r} \omega_{\psi}^{n} \big)^{ \frac{1}{r} } 
	\big( \int_{X} w^{ (2p+1) \alpha  } \omega_{\psi}^{n} \big)^{ \frac{2p}{ (2p+1) \alpha  }   }	
	\end{align*}
	where
	\begin{align*}
	2 r=\frac{ 2 (2p+1) \alpha}{  (2p+1) \alpha-2p   } < \frac{2 \alpha}{\alpha-1} <b, \quad \forall p \geq \frac{1}{2}. 
	\end{align*}
	
	Recalling that
	\begin{align*}
	\omega_{\psi}^{n} = S_{\epsilon}^{ \kappa} e^{ h-\lambda\vphi_{  t,\epsilon }   +c_{t,\epsilon} } \omega_X^{n}, 
	\end{align*}
	then  
	\begin{align*}
	\int_{X} |\nabla h|^{2r} \omega_{\psi}^{n} \leq C(c_{t,\epsilon},  \|\varphi_{t,\epsilon}\|_{L^{\infty}(X)}    ) \int_{X} |\nabla h|^{2r} e^{h}  \omega_{X}^{n}.
	\end{align*}
	
	Since $2r <b$, then by the H\"older inequality,
	\begin{align*}
	\int_{X} |\nabla h|^{2r} e^{h} \omega_{X}^{n}& \leq \int_{X} \big(  |\nabla h|_{ \omega_{X}  }^{2} tr_{\omega} \omega_{X} \big)^{r} e^{h} \omega_{X}^{n}
	\\
	& \leq C(\omega_{X},\omega_{K} ) \int_{X}|\nabla h|_{ \omega_{X}  }^{ 2r } e^{ \frac{2r}{b} h  } e^{  (1-\frac{2r}{b}) h    } \omega_{X}^{n}\\
	& \leq C(\omega_{X},\omega_{K} ) \big( \int_{X} |\nabla h|_{ \omega_{X}  }^{b} e^{h} \omega_{X}^{n} \big)^{ \frac{2r}{b}  } \big(\int_{X} e^{h} \omega_{X}^{n} \big)^{  1-\frac{2r}{b}   } \\
	&=C(\omega_{X},\omega_{K} ) b^{2r}\big( \int_{X} |\nabla e^{\frac{h}{b}}|_{\omega_{X}}^{b} \omega_{X}^{n} \big)^{ \frac{2r}{b}  } \big(\int_{X} e^{h} \omega_{X}^{n} \big)^{  1-\frac{2r}{b}   },
	\end{align*}
	which yields that
	\begin{align*}
	\mathcal{Q}_{4} \leq C(c_{t,\epsilon}, \|\varphi_{t,\epsilon}\|_{L^{\infty}(X)}  ,b, \| e^{\frac{h}{b} } \|_{ W^{1,b}( \omega_{X} )  } )
	\big( \int_{X} w^{ (2p+1) \alpha  } \omega_{\psi}^{n} \big)^{ \frac{2p}{ (2p+1) \alpha  }   }	.
	\end{align*}
\end{proof}

\subsubsection{Estimate of $\mathcal{Q}_{5}$}

\begin{prop}
	\label{prop estimate of Q5}
	(1). If $e^{ \frac{h}{b} } \in W^{1,b}(\omega_{X}) $  for some $b >2n$, then for any $\kappa$ satisfying 
	\begin{align*}
	\kappa>(\frac{1}{2}+\frac{n_{s}}{b}     )n -n_{s},
	\end{align*}
	 there exist $\alpha \in ( 1,\frac{n}{n-1} )$ depending on $\kappa,n_{s},b$ and 
	 a positive constant $\tilde{C}_{A}$ depending on $\kappa, n_{s}, b, \| e^{\frac{h}{b} } \|_{ W^{1,b}( \omega_{X}      )}, \|\varphi_{t,\epsilon}\|_{L^{\infty}(X)}, c_{t,\epsilon}  $ such that 
	\begin{align*}
	\mathcal{Q}_{5} \leq \tilde{C}_{A} \big( \int_{X}  w^{(2p+1) \alpha} \omega_{\psi}^{n} \big)^{ \frac{2p}{ (2p+1) \alpha  }  }. 
	\end{align*}
	(2). If $e^{ \frac{h}{\hat{b}}  } \in W^{1,\infty}(\omega_{X}) $ for some $\hat{b}>n$ , then for any $\kappa>\frac{n}{2}-n_{s}$, there 
	exist $\alpha \in ( 1,\frac{n}{n-1} )$ depending on $\kappa,n_{s}$ and $\hat{C}_{A}$ depending on $\kappa, n_{s},  \hat{b}, \| e^{\frac{h}{\hat{b}}}  \|_{ W^{1,\infty}( \omega_{X}      )}$, $\|\varphi_{t,\epsilon}\|_{L^{\infty}(X)}, c_{t,\epsilon}  $ such that 
	\begin{align*}
	\mathcal{Q}_{5} \leq \hat{C}_{A} \big( \int_{X} w^{(2p+1) \alpha} \omega_{\psi}^{n} \big)^{ \frac{2p}{ (2p+1) \alpha  }  }. 
	\end{align*}
\end{prop}

\begin{proof}
	\textbf{Case 1}: We assume $e^{ \frac{h}{b} } \in W^{1,b}(\omega_{X}) $  for some $b >2n$ and choose $\kappa$ satisfying 
	$\kappa+n_{s}>(\frac{1}{2}+\frac{n_{s}}{b}     )n$, then there exists $\epsilon>0$ 
	depending on $\kappa, b,n_{s}$ such that
	\begin{align*}
	\kappa+n_{s}>(\frac{1}{2}+\frac{n_{s}}{b}     )(n+\epsilon)+\epsilon. 
	\end{align*}
	By choosing $\alpha  \in (  \frac{n+\epsilon}{ n+\epsilon-1   } ,\frac{n}{n-1}   )$, 
	and applying  the H\"older inequality,
	\begin{align}
	 \label{LE17}
	\mathcal{Q}_{5}= \int_{X} w^{2p} |\nabla h|   S_{\epsilon}^{-\frac{1}{2}} \omega_{\psi}^{n} \leq \big( \int_{X}|\nabla h|^{r} S_{\epsilon}^{ -\frac{r}{2}  } \omega_{\psi}^{n} \big)^{\frac{1}{r} } \big( \int_{X}w^{(2p+1) \alpha} \omega_{\psi}^{n} \big)^{ \frac{2p}{ (2p+1) \alpha  }  },
	\end{align}
	where
	\begin{align}
	\label{LE18}
\frac{2 \alpha}{ 2 \alpha-1 }  \leq r:=\frac{  (2p+1) \alpha}{  (2p+1) \alpha-2p   } <\frac{\alpha}{\alpha-1} <n+\epsilon, \quad \forall p \geq \frac{1}{2}. 
	\end{align}
	Recalling that
	\begin{align}
	\label{LE19}
	\omega_{\psi}^{n} = S_{\epsilon}^{ \kappa} e^{ h-\lambda\vphi_{  t,\epsilon }   +c_{t,\epsilon} } \omega_X^{n}, 
	\end{align}
	it yields that
	\begin{align*}
\int_{X}|\nabla h|^{r} S_{\epsilon}^{ -\frac{r}{2}  } \omega_{\psi}^{n} & \leq C(  \|\varphi_{t,\epsilon}\|_{L^{\infty}(X)}, c_{t,\epsilon}    ) \int_{X} |\nabla h |^{r} e^{h} S_{\epsilon}^{ \kappa-\frac{r}{2} } \omega_{X}^{n} \\
& \leq C(  \|\varphi_{t,\epsilon}\|_{L^{\infty}(X)}, c_{t,\epsilon}    ) \big( \int_{X} |\nabla h|^{b} e^{ \frac{b}{r} h   } \omega_{X}^{n} \big)^{ \frac{r}{b}  } \big( \int_{X } S_{\epsilon}^{  \frac{b}{b-r} (\kappa-\frac{r}{2}) }  
\omega_{X}^{n} \big)^{1-\frac{r}{b}  } \\
& \leq C(  \|\varphi_{t,\epsilon}\|_{L^{\infty}(X)}, c_{t,\epsilon}  , b,\| e^{\frac{h}{b} } \|_{ W^{1,b}( \omega_{X}      )} ) \big( \int_{X } S_{\epsilon}^{  \frac{b}{b-r} (\kappa-\frac{r}{2}) }  
\omega_{X}^{n} \big)^{1-\frac{r}{b}  },
	\end{align*}
	where we use the facts that $r<n+\epsilon<2n<b$ and 
	$\|e^{h}\|_{C^{0}(X)}$ can be bounded by $\| e^{\frac{h}{b} } \|_{ W^{1,b}( \omega_{X}      )}$ by the Sobolev embedding theorem. 
	Observing that
	\begin{align*}
	\frac{b}{b-r}( \kappa-\frac{r}{2})+n_{s}&=\frac{b}{b-r}( \kappa-\frac{r}{2}+\frac{b-r}{b}n_{s} ) \\
	&=\frac{b}{b-r} \big( \kappa+n_{s}-(\frac{1}{2}+\frac{n_{s}}{b})r \big) \\
	& \geq \frac{b}{b-r} \big(   \kappa+n_{s} -(\frac{1}{2}+\frac{n_{s}}{b})(n+\epsilon)          \big) \\
	& \geq \frac{b}{b-1} \epsilon,
	\end{align*}
hence we obtain that
\begin{align*}
\mathcal{Q}_{5} \leq \tilde{C}_{A} \big( \int_{X}u^{(2p+1) \alpha} \omega_{\psi}^{n} \big)^{ \frac{2p}{ (2p+1) \alpha  }  },
\end{align*}
		where $\tilde{C}_{A}$ depends on $\kappa, n_{s}, b,  \|\varphi_{t,\epsilon}\|_{L^{\infty}(X)}, c_{t,\epsilon}  , \| e^{\frac{h}{b} } \|_{ W^{1,b}( \omega_{X}      )}$.

	\textbf{Case 2:} We assume if $e^{ \frac{h}{\hat{b}}  } \in W^{1,\infty}(\omega_{X}) $ for some $\hat{b}>n$ and choose $\kappa>\frac{n}{2}-n_{s}$, then there exists $\epsilon>0$ depending on $\kappa, n_{s}$ such that
	\begin{align*}
\hat{b}>n+\epsilon, \quad 	\kappa > \frac{n}{2}-n_{s}+\epsilon. 
	\end{align*}
	By choosing $\alpha  \in (  \frac{n+\epsilon}{ n+\epsilon-1   } ,\frac{n}{n-1}   )$, 
	and applying  the H\"older inequality, (\ref{LE17}) and (\ref{LE18}) still hold. 
Using (\ref{LE19}) again, we obtain
	\begin{align*}
	\int_{X}|\nabla h|^{r} S_{\epsilon}^{ -\frac{r}{2}  } \omega_{\psi}^{n} & \leq C(  \|\varphi_{t,\epsilon}\|_{ L^{\infty}(X) },c_{t,\epsilon} ) \int_{X} |\nabla h|^{r} S_{\epsilon}^{ \kappa-\frac{r}{2} } e^{h} \omega_{X}^{n} \\
	& =C( \|\varphi_{t,\epsilon}\|_{ L^{\infty}(X) } , c_{t,\epsilon}   ) \int_{X} |\nabla h|^{r} e^{ \frac{r}{ \hat{b} } h  }  S_{\epsilon}^{ \kappa-\frac{r}{2} } e^{(1-\frac{r}{ \hat{b} } ) h} \omega_{X}^{n} \\
	& \leq C(  \|\varphi_{t,\epsilon}\|_{ L^{\infty}(X) }, c_{t,\epsilon}, \sup_{X} e^{h}  ) \int_{X} |\hat{b} \nabla (e^{ \frac{h}{ \hat{b} }  } )|^{r} S_{\epsilon}^{ \kappa-\frac{r}{2} } \omega_{X}^{n} \\
	& \leq C(  \|\varphi_{t,\epsilon}\|_{ L^{\infty}(X) }, c_{t,\epsilon},\|e^{\frac{h}{ \hat{b} } }\|_{W^{1,\infty}(X)}  ) \int_{X}  S_{\epsilon}^{ \kappa-\frac{r}{2} } \omega_{X}^{n} ,
	\end{align*}
	where we use the fact that $r<n+\epsilon<\hat{b}$. 
	Since 
	\begin{align*}
	-\frac{r}{2}+\kappa+n_{s} > -\frac{1}{2} (n+\epsilon)+\frac{n}{2}+\epsilon = \frac{\epsilon}{2},
	\end{align*}
	we obtain that
\begin{align*}
\mathcal{Q}_{5} \leq \hat{C}_{A} \big( \int_{X}u^{(2p+1) \alpha} \omega_{\psi}^{n} \big)^{ \frac{2p}{ (2p+1) \alpha  }  },
\end{align*}
where $\hat{C}_{A}$ depends on $\kappa, n_{s}, \hat{b}, \|\varphi_{t,\epsilon}\|_{L^{\infty}(X)}, c_{t,\epsilon}  , \| e^{\frac{h}{\hat{b}} } \|_{ W^{1,\infty}( \omega_{X}      )}$. 
\end{proof}

\begin{rem}
	\label{remark Q5}
	If $e^{ \frac{h}{\hat{b}}  } \in W^{1,\infty}(\omega_{X}) $ for some $\hat{b}>n$ , we can not conclude that there exists $b>2n$ such that $e^{ \frac{h}{b}  } \in W^{1,b}(\omega_{X})$ since we do not assume $e^{h}$ has a priori positive lower bound. 
	\end{rem}

Theorem \ref{thm 2nd Laplacian estimate intro}  follows from (\ref{LE16}), Proposition \ref{prop estimate of Q4}, Proposition \ref{prop estimate of Q5}, Lemma \ref{lemma Sobolev inequality uniform} and Moser's iteration.

	\bibliographystyle{plain}
	\bibliography{MAsubmit}

\end{document}